\documentclass[a4paper, 11pt, twoside]{article}
\usepackage{hyperref}

\usepackage{amsmath,amssymb}
\usepackage{amsthm}
\usepackage{mathrsfs}
\usepackage[dvipdfmx]{graphicx}

\usepackage[normalem]{ulem}
\usepackage{url}
\usepackage{comment}
\usepackage{bm}
\usepackage{accents}

\usepackage{enumerate}
\usepackage[all]{xy}
\usepackage{mathtools}
\usepackage{graphicx}
\usepackage{overpic}
\usepackage{pdfpages}

\usepackage[margin=1in]{geometry}

\usepackage{fancyhdr}

\usepackage{tikz}
\usepackage{tikz-cd}
\usepackage{pgfplots}

\usetikzlibrary{arrows}
\usetikzlibrary{patterns}

\theoremstyle{definition}
\newtheorem{defi}{Definition}[section]
\newtheorem{rem}[defi]{Remark}
\newtheorem{ex}[defi]{Example}

\theoremstyle{plain}
\newtheorem{thm}[defi]{Theorem}

\newtheorem{prop}[defi]{Proposition}
\newtheorem{lem}[defi]{Lemma}
\newtheorem{cor}[defi]{Corollary}

\newtheorem{question}{Question}[section]

\renewcommand{\ker}{\operatorname{Ker}}

\renewcommand{\Im}{\operatorname{Im}}

\newcommand{\C}{\mathbb{C}}
\newcommand{\R}{\mathbb{R}}
\newcommand{\Q}{\mathbb{Q}}
\newcommand{\Z}{\mathbb{Z}}

\newcommand{\F}{\mathbb{F}}

\newcommand{\id}{\operatorname{id}}

\newcommand{\bfk}{\bold{k}}

\renewcommand{\tilde}{\widetilde}

\renewcommand{\epsilon}{\varepsilon}

\newcommand{\la}{\langle}
\newcommand{\ra}{\rangle}

\newcommand{\sign}{\operatorname{sign}}

\newcommand{\Ham}{\operatorname{Ham}}
\newcommand{\M}{\mathcal{M}}

\newcommand{\rest}[2]{\left.#1\right|_{#2}}

\newcommand{\ab}{\mathrm{ab}}

\newcommand{\knot}{\mathrm{knot}}

\newcommand\shorttitle{\footnotesize{Topological constraints on clean Lagrangian intersections from $\Q$-valued augmentations}}
\newcommand\authors{Yukihiro Okamoto}

\usepackage{authblk}

\AtEndDocument{%
  \par
  \medskip
  \begin{tabular}{@{}l@{}}%
    Department of Mathematical Sciences,
Tokyo Metropolitan University, \\
Minami-osawa, Hachioji, Tokyo, 192-0397, Japan
 \\
    \textit{E-mail address}: \texttt{yukihiro@tmu.ac.jp}
  \end{tabular}}

\begin{document}

\title{Topological constraints on clean Lagrangian intersections from $\mathbb{Q}$-valued augmentations}
\author{Yukihiro Okamoto }
\date{}
\maketitle

\begin{abstract}
\noindent
Let $K$ be a knot in $\mathbb{R}^3$ which has the $(2,q)$-torus knot for $q\neq \pm 1$ or the figure-eight knot as a component of connected sum. For its conormal bundle $L_K$ in $T^*\mathbb{R}^3$, we show that there is no compactly supported Hamiltonian diffeomorphism $\varphi$ on $T^*\mathbb{R}^3$ such that $\varphi(L_K)$ intersects the zero section $\mathbb{R}^3$ cleanly along the unknot in $\mathbb{R}^3$.
Using symplectic field theory, the proof is reduced to studying the augmentation variety $V_{\bold{k}}(K)$ of  $K$ over a filed $\bold{k}$.
The key point of this paper is finding an algebraic constraint on $V_{\bold{k}}(K)$ which is valid only when $\bold{k}$ is not algebraically closed, and the proof is completed by some arithmetic argument with $\bold{k}=\mathbb{Q}$.
\end{abstract}

\section{Introduction}\label{sec-intro}

\noindent
\textbf{Background.}
Let us consider the cotangent bundle $T^*\R^3$ of $\R^3$ with its canonical symplectic structure.
We identify the zero section with $\R^3$.
Given any knot $K$ in $\R^3$, let $L_K$ denote its conormal bundle.
It is a Lagrangian submanifold of $T^*\R^3$ and intersects $\R^3$ cleanly along the knot $K$. 
Let $\mathrm{Ham}_c(T^*\R^3)$ denote the set of Hamiltonian diffeomorphisms on $T^*\R^3$ associated to compactly supported Hamiltonians on $T^*\R^3$.
Using this notion, we ask the following question, which was observed by the author in \cite{O}:
\textit{Is there $\varphi \in \mathrm{Ham}_c(T^*\R^3)$ such that $\varphi(L_{K})$ intersects $\R^3$ cleanly along a knot whose isotopy class in $\R^3$ is different from $K$?}

As a special case, the answer to this question when $K$ is the unknot is `no'. This can be deduced from the result by Ganatra-Pomerleano \cite[Proposition 6.29]{GP} about clean Lagrangian intersections in a plumbing of two copies of the cotangent bundle of $S^3$ along the unknots (see \cite[Section 1]{O}).
Smith-Wemyss \cite[Proposition 4.10]{SW} and Asplund-Li \cite[Theorem 1.3]{AL} also studied clean Lagrangian intersections in plumbings of two copies of $T^*S^3$ along the unknots.
These results are concerned with Smith's question on the `persistence (or rigidity)
of unknottedness of Lagrangian intersections' \cite[Question 1.5]{GP}.

In a general case,
under the assumption that $\varphi(L_K)$ for $\varphi \in \Ham_c(T^*\R^3)$ intersects $\R^3$ cleanly along a knot $K'$ in $\R^3$, the author proved a theorem
\cite[Theorem 1.1]{O}, which relates the \textit{augmentation variety} in $(\C^*)^2$ of $K$ with that of $K'$.
This result leads us to the same answer to the question as above when $K$ is the unknot \cite[Theorem 1.3]{O}.
We can also deduce some constraints on $K'$ when $K$ has certain non-trivial knot types \cite[Theorem 1.4, 1.5]{O}.

In fact, the previous result \cite[Theorem 1.1]{O} is useless to solve the following question.
For the reason, see Remark \ref{rem-1} below.

\begin{question}\label{Q-2}
If $K$ is not the unknot, is there $\varphi \in \Ham_c(T^*\R^n)$ such that $\varphi(L_{K})$ intersects $\R^3$ cleanly along the unknot in $\R^3$? 
\end{question}

Note that the aforementioned persistence of unknottedness is not applicable to this question since $\varphi(L_K)$ is not the conormal bundle of a knot in general.

%

\

\noindent
\textbf{Main results.}
The goal of this paper is to give a negative answer to Question \ref{Q-2} when $K$ has a certain knot type by using augmentation variety over $\Q$ and some arithmetic argument.
Let $T_{(2,2m+1)}$ for $m\in \Z$ denote the $(2,2m+1)$-torus knot and $4_1$ denote the figure-eight knot.
%
\begin{thm}[Theorem \ref{thm-main} and Theorem \ref{thm-eight}]\label{thm-1}
Let $K$ be $K'\# T_{(2,2m+1)}$ or $K' \# 4_1$, where $K'$ is an arbitrary knot in $\R^3$ and $m\in \Z \setminus \{0,-1\}$. Then, there is no $\varphi\in \mathrm{Ham}_c(T^*\R^3)$ such that $\varphi(L_{K})$ intersects $\R^3$ cleanly along the unknot in $\R^3$.
\end{thm}

For the proof, we use Theorem \ref{thm-3} below.
For an oriented knot $K$ in $\R^3$,
the \textit{knot DGA} is defined combinatorially from a braid representation of a knot.
In this paper, we refer to a definition  in \cite{N-comb, EENS-filt, N-intro} with coefficients in $\Z[\lambda^{\pm},\mu^{\pm},U^{\pm}]$.
Fix a field $\bold{k}$.
From the set of augmentations of the knot DGA of $K$ to $\bold{k}$, following \cite[Definition 5.1]{N-intro}, we will define the \textit{augmentation variety} $V_{\bfk}(K) \subset (\bold{k}^*)^3$ in Definition \ref{def-aug-var}.
Let $(x,y,Z)$ be the coordinate of $\bold{k}^3$.
\begin{thm}[Corollary \ref{cor-aug}]\label{thm-3}
For $\varphi\in \Ham_c(T^*\R^3)$ and oriented knots $K_0$ and $K_1$ in $\R^3$,  suppose that $\varphi(L_{K_0})$ intersects $\R^3$ cleanly along $K_1$.
Then, there exist $a\in \{\pm 1\}$ and integers $b,c$ and $d$ for which we have an injective map
\[ V_{\bfk}(K_1) \hookrightarrow V_{\bfk}(K_0) \colon (x,y,Z) \mapsto (x^ay^{b}Z^{c}, y^a Z^{d} ,Z).\]
\end{thm}

This is a refinement of \cite[Theorem 1.1]{O}, and the proof is parallel to \cite{O} except an enhancement of the coefficient ring.
We use the results by Ekholm-Etnyre-Ng-Sullivan \cite{EENS, EENS-filt} which show that the Chekanov-Eliashberg DGA of the unit conormal bundle of $K_i$ is stable tame isomorphic to the knot DGA of $K_i$ for $i\in \{0,1\}$, and we
compose the stable tame isomorphisms with a DGA map associated to an exact Lagrangian cobordism.
For more on the coefficient ring, see Subsection \ref{subsubsec-CE-conormal}.

\begin{rem}\label{rem-1}
Identifying $\{(x,y,Z)\in (\bfk^*)^3 \mid Z=1\}$ with $(\bfk^*)^2$,
we obtain $V_{\bfk}(K)\cap \{Z=1\} \subset (\bfk^*)^2$, which is a version of augmentation variety used in \cite{O}.
This corresponds to the reduction of the coefficient ring of the knot DGAs from $\Z[\lambda^{\pm},\mu^{\pm},U^{\pm}]$ to $\Z[\lambda^{\pm},\mu^{\pm}]$ by setting $U=1$.
It follows from Theorem \ref{thm-3} that there exists an injective map
\[
 V_{\bfk}(K_1) \cap \{Z=1\} \hookrightarrow V_{\bfk}(K_0) \cap\{Z=1\} \colon (x,y) \mapsto (x^ay^{b},y^a)
\]
for some $a\in \{\pm1\}$ and $b\in \Z$. This is a slight refinement of \cite[Theorem 1.1]{O},
but it gives no constraint on $K_0$ if $K_1=O$, where $O$ is the unknot. The reason is that a relation  $V_{\bold{k}}(O) \cap \{Z=1\}  \subset V_{\bold{k}}(K) \cap\{Z=1\}$ automatically holds for any knot $K$.
See \cite[Section 5.2]{N}
\end{rem}


Compared with \cite{O}, where $V_{\C}(K_i) \cap \{Z=1\}\subset(\C^*)^2$ for  $i\in \{0,1\}$ were used, a novel point of this paper lies in an algebraic argument to deduce Theorem \ref{thm-1} from Theorem \ref{thm-3}. 

Let us explain the idea in the case $K=K'\# T_{(2,2m+1)}$ for $m\in \Z_{\geq 1}$.
It suffices to show that for some field $\bold{k}$, $V_{\bfk}(O)$ cannot be embedded into $V_{\bfk}(K'\# T_{(2,2m+1)})$ by the map of Theorem  \ref{thm-3}.
From the knot DGA of $O$, which is quite simple, it is easy to compute that $V_{\bfk}(O)=\{ Z- x - y + xy=0 \}$.
However, determining $V_{\bfk} (K'\# T_{(2,2m+1)})$ is hard because the knot DGA is complicated in general.
As a matter of fact, for the closure of a braid with $n$ strands, the $0$-th homology of its knot DGA has $n(n-1)$ generators and $2n^2$ relations. See Subsection \ref{subsubsec-HC0}.


To solve this problem, we use an algebraic constraint on $V_{\bfk}(K'\#T_{(2,2m+1)})$.
In Section \ref{subsec-thm1}, by focusing on specific two relations in the $0$-th homology, we will extract a polynomial $P_m\in  \Z[\mu^{\pm},U^{\pm}][T]$ defined inductively on $m$ with the following property for any $K'$:
\[ V_{\bfk}(K'\# T_{(2,2m+1)}) \subset \{ (x,y,Z)  \mid \rest{P_m}{\mu=y,U=Z} \in \bfk[T] \text{ has a root in }\bold{k} \}. \]
Then, the proof of Theorem \ref{thm-1} is reduced to finding $(x,y,Z)\in V_{\bfk}(O)$ such that the $\bfk$-coefficient polynomial $\rest{P_m}{\mu=y^aZ^d,U=Z}$ does not have a root in $\bfk$.
It is worth noting that this approach does not work if $\bold{k}$ is algebraically closed.
Actually, we take $\bfk=\Q$ and apply Hilbert's irreducibility theorem in the proof of Theorem \ref{thm-main}.

In the case  $K=K'\#4_1$, we can explicitly find a polynomial (of degree $3$ in $T$) with a similar property.
We take $\bfk=\Q$ and use the rational root theorem in the proof of Theorem \ref{thm-eight}.

\

\noindent
\textbf{Organization of paper.}
In Section \ref{sec-CE}, for a contact manifold $Y$ and its Legendrian submanifold $\Lambda$ satisfying certain conditions, we introduce the Chekanov-Eliashberg DGA of $(Y,\Lambda)$ with coefficients in $\Z[H_2(Y,\Lambda)]$. We then define a DGA map associated to an exact Lagrangian cobordism in $\R\times Y$.
In Section \ref{subsec-DGA-clean}, we will see how to apply this DGA map to a setting with clean Lagrangian intersection.
In Section \ref{sec-proof}, we combine the argument in Section \ref{subsec-DGA-clean} with the results of \cite{EENS-filt, N-intro} about the knot DGA with coefficients in $\Z[\lambda^{\pm},\mu^{\pm},U^{\pm}]$.
Theorem \ref{thm-3} is proved in Section \ref{subsec-general}.
Theorem \ref{thm-1} for $K=K' \# T_{(2,2m+1)}$ with $m\neq 0,-1$ (resp. $K=K' \# 4_1$) is proved by examining the knot DGA of $K' \# T_{(2,2m+1)}$ in Section \ref{subsec-thm1} (resp. $K' \# 4_1$ in Section \ref{subsec-thm2}).
In Appendix \ref{sec-Floer}, by using Lagrangian Floer cohomology group over $\F_2[\lambda^{\pm}]$, we prove Proposition \ref{prop-Floer}, which gives a homological constraint on $1$-dimensional clean intersections.
This strengthen a result in \cite[Section 2]{O}. Nevertheless, the present proof is much simpler than that of \cite{O}.

\

\noindent
\textbf{Convention.}
In this paper, we take $\Z$ as the coefficient ring of the singular homology groups, unless otherwise mentioned.

\

\noindent
\textbf{Acknowledgements.}
The author would like to thank Kei Irie and Kaoru Ono for making valuable comments.
He would also like to thank Fumikazu Nagasato and Kouki Sato for the discussion before starting this work and giving some insight from knot theory.
He would also like to thank Lenhard Ng for the discussion during CBMS summer school 2025 and suggestions for improving the main result in the previous version.
The author appreciates the referee's valuable comments.
This work was supported by JSPS KAKENHI Grant Number JP23KJ1238.
The author is currently supported by  JSPS KAKENHI Grant Number
JP25KJ0270.

\section{Symplectic field theory for exact Lagrangian cobordisms}\label{sec-CE}

We fix a $2n$-dimensional simply connected manifold $P$ with a Liouville $1$-form $\lambda$ such that $P$ is the completion of a compact Liouville domain $(\bar{P},\rest{\lambda}{\bar{P}})$ satisfying $2 c_1(T\bar{P})=0$.
Let us consider
$ Y\coloneqq P\times \R$ and $ \alpha \coloneqq dz-\lambda$, 
where $z$ is the coordinate of $\R$.
Then, $\alpha$ is a contact form on $Y$ and the Reeb vector field on $Y$ with respect to $\alpha$ agrees with $\partial_z$. Let $\pi_P \colon Y \to P$ denote the projection to $P$.

In this section, we present (a version of) \textit{symplectic filed theory} introduced by Eliashberg-Givental-Hofer \cite{EGH}.
The arguments from Section \ref{subsec-notation} to Section \ref{subsec-DGA-map} are parallel to those in \cite[Section 3]{O}, except the difference of the coefficient rings.
In Section \ref{subsec-DGA-clean}, we will discuss a particular case related to clean Lagrangian intersections.

\subsection{Preliminaries}\label{subsec-notation}

We fix a complex structure $j_0$ on $\R\times [0,1]$ such that $(j_0)_{(s,t)}(\partial_s)=\partial_t$, where $(s,t)$ is the coordinate of $\R\times [0,1]$.
Let $D\coloneqq \{z\in \C \mid |z| \leq 1\}$. For any $m\in \Z_{\geq 0}$, we fix $p_0,p_1,\dots ,p_m\in \partial D$ ordered counterclockwise and define $D_{m+1} \coloneqq D\setminus \{p_0,p_1, \dots ,p_m \}$. Let $\mathcal{C}_{m+1}$ denote the space of conformal structures on $D_{m+1}$ which can be extended smoothly on $D$.
Given $\kappa \in \mathcal{C}_{m+1}$, we let $j_{\kappa}$ denote a complex structure on $D_{m+1}$ which associates $\kappa$. Near each boundary puncture of the Riemann surface $(D_{m+1},j_{\kappa})$, we fix a strip coordinate by a biholomorphic map
\[\begin{array}{cc}
\psi_0 \colon [0,\infty)\times [0,1] \to D_{m+1}, &
\psi_k  \colon (-\infty ,0] \times [0,1] \to D_{m+1} \text{ for }k\in \{1,\dots ,m,\}
\end{array}\]
with respect to $j_0$ and $j_{\kappa}$ such that $\lim_{s\to \infty}\psi_0(s,t) = p_0$ and $\lim_{s\to -\infty} \psi_k(s,t) = p_k$ uniformly on $t\in [0,1]$.

Let $\overline{D_{m+1}}$ be a compactification of $D_{m+1}$ obtained by gluing $[0,\infty]\times [0,1]$ via $\psi_0$ and $[-\infty,0]\times [0,1]$ via $\psi_k$ for every $k=1,\dots ,m$.
In a natural way, $\overline{D_{m+1}}$ becomes a compact oriented surface with boundary and corners.

Let $\Lambda$ be a compact and connected Legendrian submanifold of $(Y,\alpha)$.
Denote the set of its Reeb chords by
\[\textstyle{ \mathcal{R}(\Lambda) \coloneqq \{ a\colon [0,T]\to Y \mid T>0,\  \frac{d a}{dt}(t)= (\partial_z)_{a(t)} \text{ for every }t\in [0,T] \text{ and }a(0),a(T) \in \Lambda \}. }\]
Hereafter, for any Reeb chord $(a\in [0,T]\to Y)\in \mathcal{R}(\Lambda)$, we write $T= \int a^* \alpha$ as $T_a$.
For each $a \in \mathcal{R}(\Lambda)$, we fix a continuous map $w_a\colon \overline{D_1} \to Y$ such that $w_a(\partial D_1) \subset \Lambda$ and $\lim_{s\to \infty} w_a \circ \psi_0(s,t) = a(T_at)$
for every $t\in [0,1]$.

\subsection{Moduli space of pseudo-holomorphic curves with boundary in an exact Lagrangian cobordism}\label{subsec-moduli}

Let $\Lambda_+$ and $\Lambda_-$ be two compact Legendrian submanifolds of $(Y,\alpha)$.
Recall that  a submanifold $L$ of $\R\times Y$ is called an \textit{exact Lagrangian cobordism} from $\Lambda_-$ to $\Lambda_+$ if the following hold:
\begin{itemize}
\item There exist $r_+,r_-\in \R$ with $r_-<r_+$ such that
\[\begin{array}{ll} L\cap ([r_+,\infty) \times Y) = [r_+,\infty) \times \Lambda_+, &  L \cap ((-\infty,r_-]\times Y) = (-\infty ,r_-]\times \Lambda_- , \end{array}\]
and $L \cap ([r_-,r_+]\times Y)$ is compact.
\item There exists a $C^{\infty}$ function $f\colon L\to \R$ such that $df = \rest{(e^r\alpha)}{L}$ and $f$ is constant when restricted on $(-\infty,r_-]\times \Lambda_-$.
Here, $r$ is the $\R$-coordinate of $\R\times Y$.
\end{itemize}
In the rest of this paper, we only consider the case where $\Lambda_-$ is connected. Therefore, we do not need to care about the last condition that $\rest{f}{(-\infty,r_-]\times \Lambda_-}$ is a constant function since
\[\rest{df}{(-\infty,r_-]\times \Lambda_-} = \rest{(e^r \alpha)}{(-\infty,r_-]\times \Lambda_-} =0,\]
which implies that $\rest{f}{(-\infty,r_-]\times \Lambda_-}$ is locally constant.

Associated to $L$, we use the following notation for maps between pairs of spaces:
\begin{align}\label{inclusion-pm}
\begin{array}{cc}
  i_+  \colon (Y, \Lambda_+) \to (\R\times Y, L) \colon x \mapsto (r_+,x) , &
  i_-  \colon (Y, \Lambda_-) \to (\R\times Y,  L) \colon x \mapsto (r_-,x) ,
  \end{array}
\end{align}
and
\[R_L \colon (\R\times Y,  L) \to ([r_-,r_+]\times Y, L\cap ([r_-,r_+]\times Y))\] 
is a strong deformation retract such that $R_L(r,x) = (r_+,x)$ if $r\geq r_+$ and 
$R_L(r,x) = (r_-,x)$ if $r\leq r_-$.
In addition, 
for every $s\in \R$, we define a translation map
\[\tau_s \colon \R\times Y \to \R\times Y \colon (r,x) \mapsto (r+s,x).\]

Let $J$ be an almost complex structure on $P$ such that $d\lambda(\cdot, J \cdot)$ is a Riemannian metric on $P$ and \textit{adapted} to $\Lambda$ in the sense of the definition in \cite[Section 2.3]{EES}. 
We associate to $J$ an almost complex structure $\tilde{J}$ on $\R\times Y$ determined by $\tau_s^* \tilde{J} = \tilde{J}$ for every $s\in \R$, $\tilde{J}(\ker \alpha) = \ker \alpha$,\ $(\pi_P)_*(\tilde{J}(v)) =J (  (\pi_P)_*(v))$ for every $v\in \ker \alpha$ and $\tilde{J}(\partial_z) = \partial_r$.

For any $a\in \mathcal{R}(\Lambda_+)$, $b_1,\dots ,b_m\in \mathcal{R}(\Lambda_-)$ and $C \in H_2(\R\times Y, L)$, we define
$\hat{\M}_{L,C}(a;b_1,\dots ,b_m)$
to be the space of pairs $(u,\kappa)$ of $\kappa \in \mathcal{C}_{m+1}$ and a $C^{\infty}$ map $u\colon D_{m+1} \to \R\times Y$ such that:
\begin{itemize}
\item $d u + \tilde{J}\circ du \circ j_{\kappa} =0$.
\item $u( \partial D_{m+1}) \subset L$.
\item There exist $s_0,s_1,\dots ,s_m \in \R$ such that
\[\begin{array}{cc}
 \lim_{s\to \infty} \tau_{-s} \circ u \circ \psi_0  = (s_0, a(T_a(1-t)) ) , &
  \lim_{s\to -\infty} \tau_{s} \circ u \circ \psi_k   = (s_k, b_k(T_{b_k}(1-t)) ) 
 \end{array}\]
for $k\in \{1,\dots ,m\}$, $C^{\infty}$-uniformly on $[0,1]$. 
\item 
Let $\bar{u}\colon \overline{D_{m+1}} \to L$ denote the continuous extension of the map $R_L \circ u \colon D_{m+1} \to L$.
Then, the 2-chain 
\[\bar{u} + (i_+\circ w_a) -( (i_-\circ w_{b_1}) + \dots + (i_-\circ w_{b_m})) \]
represents the homology class $C$. (More precisely, each term is regarded as a singular 2-chain after taking suitable triangulations of $\overline{D_{m+1}}$ and $\overline{D_1}$.)
\end{itemize}
When $m \leq 1$, $\mathcal{C}_{m+1}$ consists of a single element and the group $\mathcal{A}_{m+1}$ of conformal automorphisms on $D_{m+1}$ acts on $\hat{\M}_{L,C}(a;b_1,\dots ,b_m)$ by $(u,\kappa)\cdot \varphi \coloneqq  (u\circ \varphi ,\kappa)$ for every $\varphi \in \mathcal{A}_{m+1}$. 
Let us define a moduli space to be
\[ \M_{L,C} (a;b_1,\dots ,b_m) \coloneqq \begin{cases} 
\hat{\M}_{L,C}(a;b_1,\dots ,b_m) / \mathcal{A}_{m+1} & \text{ if } m=0,1, \\
\hat{\M}_{L,C}(a;b_1,\dots ,b_m) & \text{ if } m\geq 2.
\end{cases}\]

In the case where $\Lambda_+=\Lambda_- = \Lambda$ and $L = \R\times \Lambda$, for any $a,b_1,\dots ,b_m\in \mathcal{R}(\Lambda)$ and $A\in H_2(Y,\Lambda)$, an $\R$-action on $\M_{\R\times \Lambda , (i_+)_*(A)}(a;b_1,\dots ,b_m)$ is defined by $s \cdot (u,\kappa)\coloneqq (\tau_{s}\circ u, \kappa)$ for every $s\in \R$. Let us denote the quotient space by
\[ \bar{\M}_{\Lambda,A}(a;b_1,\dots ,b_m) \coloneqq  \R\backslash \M_{\R\times \Lambda , (i_+)_*(A)}(a;b_1,\dots ,b_m).\]

\subsection{Chekanov-Eliashberg DGA over $\Z[H_2(Y,\Lambda)]$}\label{subsec-CE-DGA}

We consider a compact Legendrian submanifold $\Lambda$ of $(Y,\alpha)$ satisfying the following conditions:
\begin{enumerate}
\item $\Lambda$ is a connected and oriented manifold and admits a spin structure.
\item For the immersion $\rest{\pi_P}{\Lambda} \colon \Lambda \to P$, the self-intersection of $\pi_P(\Lambda)$ consists only of transverse double points.
\item For the immersed Lagrangian submanifold $\pi_P(\Lambda)$ of $P$, its Maslov class  vanishes.
\end{enumerate}
The second condition is satisfied for generic $\Lambda$.
Under the second condition, the Conley-Zehnder index $\mu(a) \in \Z$ of $a$ is defined by using the capping path $\rest{w_a}{\partial D_1}$ for $a$ as in \cite[Section 3.1]{O}. 
We then define
$|a|\coloneqq \mu(a) -1$.
We also assume that $\Lambda$ is \textit{admissible} in the sense of the definition in \cite[Section 2.3]{EES}. 
Let us fix a spin structure $\mathfrak{s}$ on $\Lambda$.

We take an almost complex structure $J$ on $P$ as in \cite[Proposition 2.3]{EES} such that the following hold for every $a,b_1,\dots ,b_m\in \mathcal{R}(\Lambda)$ and $A\in H_2(Y,\Lambda)$ such that $(b_1,\dots ,b_m)\neq (a)$:
The moduli space $\bar{\M}_{\Lambda,A}(a;b_1,\dots ,b_m)$ is cut out transversely and it is a manifold of dimension $|a|-(|b_1|+\dots +|b_m|)-1$. Moreover, when $|b_1|+\dots + |b_m|=|a|-1$, the moduli space
\[\coprod_{A\in H_2(Y,\Lambda)}\bar{\M}_{\Lambda,A}(a;b_1,\dots ,b_m)\]
is a compact $0$-dimensional manifold.
(For the relation between $J$-holomorphic disks in $P$ defined as in \cite[Section 2.3]{EES} and $\tilde{J}$-holomorphic maps in $\R\times Y$ defined in Section \ref{subsec-moduli}, see \cite[Theorem 2.1]{DR} and \cite[Appendix A.4]{O}.)

In addition, $\bar{\M}_{\Lambda,A}(a;b_1,\dots ,b_m)$ has an orientation depending on $\mathfrak{s}$ and a choice of an orientation of the determinant line of the capping operator for each Reeb chord of $\Lambda$.
See \cite[Definition A.4]{O} for details which follow the conventions in \cite{EES-ori} related to orientations.

Let $\mathcal{A}_*(\Lambda)$ be the unital graded $\Z[H_2(Y,\Lambda)]$-algebra freely generated by $\mathcal{R}(\Lambda)$ whose grading is given by $|\cdot|\colon \mathcal{R}(\Lambda) \to \Z$.
A derivation
$\partial_{J,\mathfrak{s}} \colon \mathcal{A}_*(\Lambda) \to \mathcal{A}_{*-1}(\Lambda)$ over $\Z [H_2(Y,\Lambda)]$
is defined for every $a\in \mathcal{R}(\Lambda)$ by
\[\partial_{J,\mathfrak{s}} (a)\coloneqq \sum_{|b_1|+\dots + |b_m|=|a|-1} \sum_{A\in H_2(Y,\Lambda)} (-1)^{(n-1)(|a|+1)} \left( \#_{\sign} \bar{\M}_{\Lambda, A}(a;b_1,\dots ,b_m) \right) e^A \cdot b_1\cdots b_m  \]
and extended by the Leibniz rule.
Here, $\#_{\sign}$ is the number of the points counted with the
signs induced by the orientation.

We refer to \cite[Theorem 4.1, Theorem 4.12]{EES-ori} which takes $\Z[H_1(\Lambda)]$ as the coefficient ring of the Chekanov-Eliashberg DGA. We modify the arguments in their proof to the present case where we take $\Z[H_2(Y,\Lambda)]$ as the coefficient ring.
Then, we obtain the following theorem.

\begin{thm}\label{thm-inv-hmgy}
$\partial_{J,\mathfrak{s}}$ is a differential, that is, $\partial_{J,\mathfrak{s}}\circ \partial_{J,\mathfrak{s}} =0$. Moreover,
the stable tame isomorphism class over $\Z[H_2(Y,\Lambda)]$ of the DGA $(\mathcal{A}_*(\Lambda),\partial_{J,\mathfrak{s}})$ is independent of the choice of $J$ and invariant under Legendrian isotopies of $\Lambda$ preserving spin structures.
\end{thm}

Let us discuss the effect by changing the spin structure on $\Lambda$.
For any two spin structures $\mathfrak{s}_0$ and $\mathfrak{s}_1$ on $\Lambda$,
a cohomology class
$ d(\mathfrak{s}_0,\mathfrak{s}_1) \in H^1(\Lambda;\Z/2)$ is defined as in \cite[Subsection 4.4.1]{EES-ori}.
Let $\partial \colon H_2(Y,\Lambda) \to H_1(\Lambda)$ denote the boundary map.
For any $\beta \in H_1(\Lambda)$, if we change the choice of a spin structure from $\mathfrak{s}$ to $\mathfrak{s}'$, the orientation on the moduli space
\[ \coprod_{|b_1|+\dots + |b_m|=|a|-1} \coprod_{A\in H_2(Y,\Lambda) \colon \partial A = \beta} \bar{\M}_{\Lambda,A}(a;b_1,\dots ,b_m)\]
is changed by the sign $(-1)^{ \la d(\mathfrak{s},\mathfrak{s}') ,\beta\ra }$. See the proof of \cite[Theorem 4.29]{EES-ori}. We define a ring isomorphism
\begin{align}\label{change-spin}
\psi(\mathfrak{s},\mathfrak{s}') \colon \Z[H_2(Y,\Lambda)] \to \Z[H_2(Y,\Lambda)] \colon e^A \mapsto (-1)^{ \la d(\mathfrak{s},\mathfrak{s}) ,\partial A\ra } e^A .
\end{align}
Then, the next proposition holds as in the proof of \cite[Theorem 4.30]{EES-ori}.
\begin{prop}\label{prop-spin}
We define a graded ring isomorphism $I(\mathfrak{s},\mathfrak{s}') \colon \mathcal{A}_*(\Lambda) \to \mathcal{A}_*(\Lambda)$ by
\[
I(\mathfrak{s},\mathfrak{s}') (k\cdot a) = \psi(\mathfrak{s},\mathfrak{s}')(k)\cdot a
\]
for every $k\in \Z[H_2(Y,\Lambda)]$ and $a\in \mathcal{R}(\Lambda)$. Then
 $\partial_{J,\mathfrak{s}'}\circ I(\mathfrak{s},\mathfrak{s}') = I(\mathfrak{s},\mathfrak{s}')  \circ \partial_{J,\mathfrak{s}}$.
\end{prop}

\subsection{DGA map associated to an exact Lagrangian cobordism}\label{subsec-DGA-map}

We consider an exact Lagrangian cobordism $L$ from $\Lambda_-$ to $\Lambda_+$ as in Section \ref{subsec-moduli} such that the Maslov class of $L$ vanishes and both $\Lambda_+$ and $\Lambda_-$ are connected.
We assume that $L$ is oriented and has a spin structure $\mathfrak{s}_L$.
Then, on the positive end $\Lambda_+$ (resp. the negative end $\Lambda_-$) of $L$, an orientation and a spin structure $\mathfrak{s}_+$ (resp. $\mathfrak{s}_-$) are induced. See \cite[Remark 3.6]{O} for their precise descriptions.

Associated to $(\Lambda_+,\mathfrak{s}_+)$ and $(\Lambda_-,\mathfrak{s}_-)$, the Chekanov-Eliashberg DGAs
$
(\mathcal{A}_*(\Lambda_+) ,\partial_{J,\mathfrak{s}_+})$ and $ (\mathcal{A}_*(\Lambda_-) ,\partial_{J,\mathfrak{s}_-})
$
are defined as in Section \ref{subsec-CE-DGA}.
The maps $i_+$ and $i_-$ of (\ref{inclusion-pm}) induce ring homomorphisms
\[\begin{array}{ll}   (i_+)_* \colon \Z[H_2(Y,\Lambda_+)] \to \Z[H_2(\R\times Y, L)] , &
(i_-)_* \colon \Z[H_2(Y,\Lambda_-)] \to \Z[H_2(\R\times Y, L)].
\end{array}\]
By $(i_-)_*$, we obtain a DGA over $\Z[H_2(\R\times Y, L)]$
\[ (\mathcal{A}_*(\Lambda_-)\otimes_{\Z[H_2(Y, \Lambda_-)]} \Z[H_2(\R\times Y, L)] , \partial_{J,\mathfrak{s}_-} \otimes \id_{\Z[H_2(\R\times Y, L)]})\]
and we regard it as a DGA over $\Z[H_2(Y,\Lambda_+)]$ by $(i_+)_*$.

The moduli space $\M_{L,C}(a;b_1,\dots ,b_m)$ is cut out transversely for generic $J$ and it is a manifold of dimension $|a|-(|b_1|+\dots +|b_m|)$ for  $a\in \mathcal{R}(\Lambda_+)$, $b_1,\dots ,b_m \in \mathcal{R}(\Lambda_-)$ and $C\in H_2(\R\times Y,L)$.
Moreover, by the Gromov compactness, if $|b_1|+\dots + |b_m|=|a|$, the moduli space
\[\coprod_{C\in H_2(\R\times Y, L)} \M_{L,C}(a;b_1,\dots ,b_m)\]
is a compact $0$-dimensional manifold.
For the proof, see \cite[Lemma 3.8]{EHK} and \cite[Section 3.3.3]{DR-surgery}, 
In addition, it has an orientation depending on $\mathfrak{s}_L$ and a choice of an orientation of the determinant line of the capping operator for each Reeb chord of $\Lambda_+$ and $\Lambda_-$. 
See \cite[Section A.3]{O}.

We refer to the proof of  \cite[Theorem 3.7]{O} which takes $\Z[H_1(\Lambda_{\pm})]$ as the coefficient rings.
We modify the arguments in \cite[Appendix A.6]{O} to the present case where we take $\Z[H_2(Y,\Lambda_{\pm})]$ as the coefficient rings.
(As referred in \cite[Appendix A]{O}, see also \cite[Theorem 2.5]{K} about the DGA map over $\Z$, but note that some conventions in \cite{K} related to orientations are different from those in \cite{EES-ori, O}.)
Then, we obtain the following theorem.

\begin{thm}\label{thm-DGA-map}
We define a unital graded $\Z[H_2(Y , \Lambda_+)]$-algebra map
\[\Phi_L \colon \mathcal{A}_*(\Lambda_+) \to
\mathcal{A}_*(\Lambda_-) \otimes_{\Z[H_2(Y,\Lambda_-)]} \Z[H_2(\R\times Y , L)]\]
such that for every $a\in \mathcal{R}(\Lambda_+)$,
\[ \Phi_L (a) = \sum_{|b_1|+\dots + |b_m| = |a| } \sum_{ C \in H_2(\R\times Y , L) } (-1)^{(n-1)(|a|+1)+m} \left( \#_{\sign} \M_{L,C} (a;b_1,\dots ,b_m) \right) b_1\cdots b_m \otimes e^{C} . \]
Then, $\Phi_L$ is a DGA map, that is, $\Phi_L\circ \partial_{J,\mathfrak{s}_+} = (\partial_{J,\mathfrak{s}_-}\otimes \id_{\Z[H_2(\R\times Y,L)]})\circ \Phi_L$.
\end{thm}

\subsection{Application to clean Lagrangian intersections}\label{subsec-DGA-clean}

\subsubsection{Chekanov-Eliashberg DGA of the unit conormal bundle of  a knot in $\R^3$}\label{subsubsec-CE-conormal}

Let $(q_1,q_2,q_3)$ be the coordinate of $\R^3$.
The manifold $\R^3$ is equipped with a standard Riemannian metric $\la \cdot , \cdot \ra \coloneqq \sum_{i=1}^3 dq_i \otimes dq_i$ 
and an orientation such that $dq_1\wedge dq_2\wedge dq_3$ is a positive volume form.
Let $(p_1,p_2,p_3)$ be the coordinate of the fiber.
We denote
\begin{align*}
D^*_{r} \R^3 & \coloneqq \{(q,p) \in T^*\R^3 \mid |p| \leq r\} \text{ for }r>0, \\
S^*\R^3 & \coloneqq \{(q,p)\in T^*\R^3 \mid |p|=1\}.
\end{align*}
The canonical Liouville form on $T^*\R^3$ is given by $\lambda_{\R^3} = \sum_{i=1}^3 p_i dq_i$.

The unit cotangent bundle 
$S^*\R^3$ is endowed with a contact form $\alpha_{\R^3} \coloneqq \rest{\lambda_{\R^3}}{S^*\R^3}$.
Note that there exists a diffeomorphism
\[S^*\R^3 \to T^*S^2 \times \R \colon (q,p) \mapsto ( (p, q- \la q, p\ra p), \la q, p\ra )\]
for the unit sphere $S^2 \subset \R^3$ such that $\alpha_{\R^3}$ on $S^*\R^3$ agrees with the pullback of the $1$-form $dz-\lambda_{S^2}$ on $ T^*S^2 \times \R$, where $\lambda_{S^2}$ is the canonical Liouville form on $T^*S^2$. Therefore,
by taking $(T^*S^2 ,\lambda_{S^2})$ as $(P,\lambda)$, we can define the Chekanov-Eliashberg DGA for Legendrian submanifolds of $S^*\R^3$ as  in Section \ref{subsec-CE-DGA}. 

We focus on a special kind of Legendrian submanifolds.
For any knot $K$ in $\R^3$, let $L_K$ be its conormal bundle. Then, the \textit{unit conormal bundle} of $K$ is defined by
\[\Lambda_K \coloneqq L_K \cap S^*\R^3 = \{(q,p) \in S^*\R^3 \mid q\in K,\  \la p, v \ra =0 \text{ for every }v\in T_q K\}. \]
It is a Legendrian submanifold of $S^*\R^3$.

Let us prepare some notation.
We identify a disk bundle of $L_K$ with a tubular neighborhood of $K$ in $\R^3$ via the map $L_K\cap D^*_{\epsilon}\R^3 \to \R^3 \colon (q,p) \mapsto q+p$ for sufficiently small $\epsilon>0$.
Then, an orientation of $L_K$ is determined from the orientation of $\R^3$. We also fix a spin structure $\mathfrak{s}_{L_K}$ given by a trivial $\mathrm{Spin}(3)$-bundle on $L_K$. As the positive end of $L_K $, an orientation and a spin structure $\mathfrak{s}_K$ on $\Lambda_K$ are determined as in \cite[Remark 3.6]{O}.

We fix an orientation of $K$. 
Let $l\colon K \to \Lambda_K$ be a section which gives the zero framing of $K$ and $m\colon S^1 \to \Lambda_K$ be the meridian of $K$. The conventions are as follows: $S^1$ is oriented as the boundary of the unit disk $D$, and $(\nu , dl(v_l) , dm(v_m) )$ is a positive basis of $T_q\R^3$ at $q\in l(K)\cap m(S^1)$, where $\nu\in T_q ( L_K\cap D^*_1\R^3)$ points outward and $v_l$ and $v_m$ are positive vectors tangent to $K$ and $S^1$ respectively.
Then, the homology classes $[l]$ and $[m]$ give a basis of $H_1(\Lambda_K)$.

We choose a non-vanishing vector field $V$ on $K$. Then, $\Lambda_K$ is bounded by a compact orientable manifold
\[\{(q, p)\in S^*\R^3 \mid q\in K,\ p(V_q)\geq 0 \} .\]
Therefore, the homology class $[\Lambda_K] \in H_2(S^*\R^3)$ vanishes. For the pair $(S^*\R^3 , \Lambda_K)$, we have a short exact sequence
\[ \xymatrix{
H_2(\Lambda_K) \ar[r]^-{0} & H_2(S^*\R^3) \ar[r] & H_2(S^*\R^3 , \Lambda_K) \ar[r]^-{\partial} & H_1(\Lambda_K) \ar[r] & H_1(S^*\R^3)=0.
}\]
We choose a basis $\{\widehat{l},\widehat{m},\widehat{s}\}$ of $H_2(S^*\R^3,\Lambda_K)$ such that $\partial( \widehat{l}) =[l]$, $\partial (\widehat{m})=[m]$ and $\widehat{s}$ is represented by the fiber $S^*_q\R^3$ of any $q\in \R^3$.
For this choice, we refer to the remark after the proof of Theorem 1.2 in \cite[Section 3.7]{EENS-filt}, using a knot which is isotopic to $K$ and transverse to the contact structure $\ker (dq_3 - q_1dq_2 + q_2 dq_1)$ on $\R^3$.
\begin{rem}
From the above short exact sequence, the lifts of $[l]$ and $[m]$ to $H_2(S^*\R^3,\Lambda_K)$ can only be well-defined modulo $\Z \cdot \widehat{s}$.
For the proof of the main results in Section \ref{sec-proof}, it is not important how to specify the lifts $\widehat{l}$ and $\widehat{m}$ (see Theorem \ref{thm-map-HC}).
\end{rem}
We denote $\lambda\coloneqq e^{\widehat{l}}$, $\mu\coloneqq e^{\widehat{m}}$ and $U\coloneqq e^{\widehat{s}}$, then
\[\Z[H_2(S^*\R^3, \Lambda_K)] = \Z[\lambda^{\pm},\mu^{\pm}, U^{\pm}].\]
From the arguments in Section \ref{subsec-CE-DGA}, the Chekanov-Eliashberg DGA $(\mathcal{A}_*(\Lambda_K) , \partial_{J,\mathfrak{s}_K})$ for $(\Lambda_K, \mathfrak{s}_K)$ is defined with coefficients in $\Z[\lambda^{\pm},\mu^{\pm},U^{\pm}]$.

\subsubsection{Exact Lagrangian cobordism arising from clean intersection}\label{subsubsec-clean}

Given a $C^{\infty}$ Hamiltonian $H\colon T^*\R^3 \times [0,1] \to \R$ with  compact support, we define $(\varphi^t_H)_{t\in [0,1]}$ to be the flow of a time-dependent vector field $(X_H^t)_{t\in [0,1]}$ on $T^*\R^3$ characterized by $d\lambda_{\R^3} (\cdot , X^t_{H}) = d(H(\cdot ,t)) $ for every $t\in [0,1]$.
We denote
\[\mathrm{Ham}_c(T^*\R^3) \coloneqq \{ \varphi = \varphi^1_H \mid H\colon T^*\R^3 \times [0,1] \to \R \text{ is a compactly supported }C^{\infty} \text{ Hamiltonian}\}.\]
The image of the zero section of $T^*\R^3$ is identified with $\R^3$.

Let us recall the definition of a clean intersection of submanifolds.
\begin{defi}
Let $L$ and $L'$ be submanifolds of $T^*\R^3$. We say that $L$ and $L'$ have a \textit{clean intersection} along $L\cap L'$ if $L\cap L'$ is a submanifold of both $L$ and $L'$, and $T_x(L\cap L') = T_x L \cap T_x L'$ for every $x\in L\cap L'$.
\end{defi}

Let us make a setup.
Let $K_0$ and $K_1$ be oriented knots in $\R^3$ and let $\varphi \in \mathrm{Ham}_c(T^*\R^3)$.
Suppose that $\varphi(L_{K_0})$ and $\R^3$ have a clean intersection along $K_1$. By moving $\varphi(L_{K_0})$ by a Hamiltonian isotopy supported in a neighborhood of $\varphi(L_{K_0}) \cap \R^3 =K_1$, we may assume that there exists $\epsilon>0$ such that
$\varphi(L_{K_0}) \cap D_{\epsilon}^*\R^3 = L_{K_1} \cap D^*_{\epsilon}\R^3$.
For the proof, see \cite[Lemma 2.3]{O}.
We consider a diffeomorphism
\[F\colon \R\times S^*\R^3 \to T^*\R^3 \setminus \R^3 \colon (r, (q,p)) \mapsto (q,e^r p) ,\]
for which $F^* (\lambda_{\R^3}) =e^r \alpha_{\R^3}$. Furthermore,
$L_{\varphi} \coloneqq F^{-1}(\varphi(L_{K_0}) \setminus K_1)$
is an exact Lagrangian cobordism from $\Lambda_{K_1}$ to $\Lambda_{K_0}$.
We have the inclusion maps (\ref{inclusion-pm})
\[ \begin{array}{cc}  i_+\colon (S^*\R^3,\Lambda_{K_0}) \to (\R\times S^*\R^3 , L_{\varphi}) , &  i_-\colon (S^*\R^3,\Lambda_{K_1}) \to (\R\times S^*\R^3 , L_{\varphi}) .\end{array}
\]

The Lagrangian filling $\varphi(L_{K_0})$ of $\Lambda_{K_0}$ in $T^*\R^3$ is obtained by gluing the Lagrangian cobordism $L_{\varphi}$ in $\R\times S^*\R^3$ and the Lagrangian filling $L_{K_1}$ in $T^*\R^3$ along $\Lambda_{K_1}$.
Therefore, as a part of the Mayer-Vietoris exact sequence, we have a short exact sequence
\[ \xymatrix{ 0= H_3(T^*\R^3 , \varphi(L_{K_0}))  \ar[r] & H_2(S^*\R^3 , \Lambda_{K_1}) \ar[r] & {\begin{matrix} H_2(\R\times S^*\R^3, L_{\varphi}) \\ \oplus  H_2(T^*\R^3 , L_{K_1}) \end{matrix} } \ar[r]^-{u} & H_2(T^*\R^3 , \varphi(L_{K_0}))  .
}\]
Note that $H_2(T^*\R^3, L_{K_1}) \cong H_1 (L_{K_1}) = \Z \cdot [K_1]$ and $H_2(T^*\R^3, \varphi(L_{K_0})) \cong H_1(L_{K_0}) = \Z \cdot [K_0]$.
The map $u$ is a surjection since $(i_+)_*(\widehat{l}) \in H_2(\R \times S^*\R^3,L_{\varphi})$ is mapped to a generator of $H_2(T^*\R^3, \varphi(L_{K_0}))$.
It follows that
$H_2(\R\times S^*\R^3 , L_{\varphi})$ is a free abelian group of rank $3$ with a basis $\{\widehat{l}_0,\widehat{m}_0,\widehat{s}_0 \}$ defined by
\[
\begin{cases}
\widehat{l}_0  \coloneqq  (i_+)_*(\widehat{l})  & \text{for } \widehat{l}\in H_2(S^*\R^3, \Lambda_{K_0}) ,\\
\widehat{m}_0  \coloneqq (i_-)_*(\widehat{m}) & \text{for } \widehat{m} \in H_2(S^*\R^3, \Lambda_{K_1}), \\
\widehat{s}_0  \coloneqq [\{0\}\times S^*_q\R^3] & \text{for any }q\in \R^3.
\end{cases}\]
We abbreviate
\[\begin{array}{ccc}
 R_+\coloneqq \Z[ H_2( S^*\R^3, \Lambda_{K_0}) ] , & R_-\coloneqq \Z[ H_2(S^*\R^3, \Lambda_{K_1})] , & R_0 \coloneqq \Z[H_2(\R\times S^*\R^3 , L_{\varphi})]  .
 \end{array}\]
The maps $i_+$ and $i_-$ induce ring homomorphisms
$ (i_+)_* \colon R_+ \to R_0$ and $(i_-)_* \colon R_- \to R_0 $. 
Let us denote
\[\begin{array}{ccc}
\lambda_0\coloneqq e^{\widehat{l}_0} = (i_+)_*(\lambda), & \mu_0\coloneqq e^{\widehat{m}_0}= (i_-)_*(\mu) , & U_0\coloneqq e^{\widehat{s}_0} = (i_+)_*(U) = (i_-)_*(U). \end{array}\]
Then $R_0 = \Z [(\lambda_0)^{\pm},(\mu_0)^{\pm} , (U_0)^{\pm}]$.

\begin{lem}\label{lem-i-pm}
$(i_+)_*\colon R_+ \to R_0$ and $(i_-)_* \colon R_- \to R_0$ are ring isomorphisms. Moreover, there exist $a\in \{\pm 1\}$ and $b,c,d \in \Z$ such that the isomorphism $(i_-)_*^{-1}\circ (i_+)_*$ is given by
\begin{align}\label{isom-i-pm}
(i_-)_*^{-1}\circ (i_+)_* \colon R_{+} \to R_- \colon \begin{cases} \lambda \mapsto \lambda^a\mu^bU^c , \\
\mu \mapsto \mu^a U^d , \\
U \mapsto U. \end{cases} 
\end{align}
\end{lem}

To prove this lemma, we use the following result.
\begin{prop}\label{prop-Floer}
The open embedding $\varphi^{-1}\colon L_{K_1} \cap D^*_{\epsilon} \R^3 \to L_{K_0}$ preserves the orientations of $L_{K_1}$ and $L_{K_0}$. Moreover,
$(\rest{\varphi}{L_{K_0}})_*([K_0]) = a\cdot [K_0]$ in $H_1(\varphi(L_{K_0}))$ for $a\in \{1,-1\}$.
\end{prop}
The proof is left to Appendix \ref{sec-Floer}.
Note that the first assertion about orientations was shown in \cite[Lemma 4.4]{O} and a version of the second assertion with coefficients in $\F_2$ was shown by \cite[Proposition 2.8, Lemma 4.5]{O} using Floer cohomology group over $\F_2$ together with the triangle product.
It should be possible to prove Proposition \ref{prop-Floer} in a way parallel to \cite[Section 2.3]{O} by lifting the coefficient ring from $\F_2$ to $\Z$.
However, we will give a simpler proof in Appendix \ref{sec-Floer} by using Floer cohomology with coefficients in $\F_2[H_1(L_{K_0})] \cong \F_2[\lambda^{\pm}]$.

\begin{proof}[Proof of Lemma \ref{lem-i-pm}]
It is clear that $(i_{\pm})_*\colon R_{\pm} \to R_0$ maps $U$ to $U_0$.
The circle $\varphi(K_0)$ in $\varphi(L_{K_0})$ is oriented via the diffeomorphism $\rest{\varphi}{K_0} \colon K_0\to \varphi(K_0)$.
From Proposition \ref{prop-Floer}, we take $a\in \{\pm1\}$ such that $[K_1] = a\cdot [\varphi(K_0)]$ in $H_1(\varphi(L_{K_0}))$. 
Consider the map induced by $i_{\pm}$ on the first homology groups
$ (i_+)_* \colon H_1(\Lambda_{K_0}) \to H_1(L_{\varphi})$ and 
$(i_-)_*\colon H_1(\Lambda_{K_1}) \to H_1(L_{\varphi})$.
A basis of $H_1(L_{\varphi})$ is given by $\{l_0,m_0\}$ for  $l_0\coloneqq (i_+)_*([l])$ and $m_0\coloneqq (i_-)_*([m])$.
It suffices to show that the equations
\begin{align}\label{homology-class}
\begin{array}{cc}
 (i_-)_*([l]) = a\cdot l_0 + b \cdot m_0 , & (i_+)_*([m]) = a \cdot m_0, \end{array}\end{align}
hold for some $b\in \Z$.

By the map $H_1(L_{\varphi}) \to H_1(\varphi(L_{K_0}))$ induced by $\rest{F}{L_{\varphi}} \colon L_{\varphi} \to \varphi(L_{K_0})$, $l_0$ is mapped to $[\varphi(K_0)]$ and $m_0$ is mapped to $0$. Since $(\rest{F}{L_{\varphi}}\circ i_-)_*([l]) = [K_1] = a\cdot [\varphi(K_0)]$ in $H_1(\varphi(L_{K_0}))$, the coefficient of $l_0$ in $(i_-)_*([l])$ is $a$. This shows that the first equation of (\ref{homology-class}) holds for some $b\in \Z$.

Next, take an embedded disk $D \subset \varphi(L_{K_0})$ which bounds the circle $ (F\circ i_+ \circ m) (S^1) $. $D$ is oriented so that
$[\partial D] = (\rest{F}{L_{\varphi}})_*( (i_+)_*([m]))$ in $H_1( \varphi(L_{K_0})\setminus K_1)$
in terms of the diffeomorphism $\rest{F}{L_{\varphi}} \colon L_{\varphi} \to \varphi(L_{K_0})\setminus K_1$.
We may assume that $D$ intersects $K_1$ transversely.
Then, their algebraic intersection number is equal to $a\in\{\pm 1\}$.
For sufficiently small $\epsilon>0$, 
\[D\cap (L_{K_1}\cap D^*_{\epsilon}\R^3) = \coprod_{i \in I}D'_i,\]
where $D'_i\subset D$ is a small disk and $I$ is a finite set.
The summand of $[\partial D_i]\in H_1(\varphi(L_{K_0})\setminus K_1)$ for all $i\in I$ satisfy
\[
 \sum_{i\in I}[\partial D_i]  = a \cdot (\rest{F}{L_{\varphi}} \circ i_-)_*([m]) = a\cdot (\rest{F}{L_{\varphi}})_*( m_0) \text{ in }H_1(\varphi(L_{K_0})\setminus K_1).
\]
For the first equality, note that the inclusion map $L_{K_1}\cap D^*_{\epsilon}\R^3 \to \varphi(L_{K_0})$ preserves orientations by Proposition \ref{prop-Floer}
when $\varphi(L_{K_0})$ is oriented via the diffeomorphism $\rest{\varphi}{L_{K_0}} \colon L_{K_0} \to \varphi(L_{K_0})$.
Let $\Sigma$ be the closure of $(\rest{F}{L_{\varphi}})^{-1}(D\setminus \coprod_{i\in I} D'_i)$ in $L_{\varphi}$. Then,
\[\begin{array}{cc}
0=[\partial \Sigma] = (\rest{F}{L_{\varphi}})_*^{-1}([\partial D]) - \sum_{i\in I} (\rest{F}{L_{\varphi}})^{-1}_*([\partial D_i]) = (i_+)_* ([m]) - a\cdot m_0 & \text{in } H_1(L_{\varphi})
\end{array}\]
This shows the second equation of (\ref{homology-class}).
\end{proof}

The exact Lagrangian cobordism $L_{\varphi}$ from $\Lambda_{K_1}$ to $\Lambda_{K_0}$ induces a DGA map as below.
\begin{prop}\label{prop-DGA-map-conormal}
There exists a unital DGA map over $R_+$
\[ \Phi_{L_{\varphi}} \colon (\mathcal{A}_*( \Lambda_{K_0}) , \partial_{J, \mathfrak{s}_{K_0}}) 
\to (\mathcal{A}_*(\Lambda_{K_1} ) , \partial_{J, \mathfrak{s}_{K_1}}) ,\]
where the $R_-$-algebra $\mathcal{A}_*(\Lambda_{K_1} ) $ is changed to an $R_+$-algebra via the isomorphism (\ref{isom-i-pm}).
\end{prop}
\begin{proof}
$L_{K_0}$ is equipped with the spin structure $\mathfrak{s}_{L_{K_0}}$ given by a trivial $\mathrm{Spin}(3)$-bundle. 
We restrict $\mathfrak{s}_{L_{K_0}}$ on $L_{K_0}\setminus \varphi^{-1}(K_1)$. Then via the diffeomorphism $F^{-1}\circ \varphi\colon L_{K_0}\setminus \varphi^{-1}(K_1) \to L_{\varphi}$, we obtain a spin structure $s_{L_{\varphi}}$ on the exact Lagrangian cobordism $L_{\varphi}$. This induces $\mathfrak{s}_{K_0}$ on the positive end $\Lambda_{K_0}$ and $\mathfrak{s}_{K_1}$ on the negative end $\Lambda_{K_0}$.
(Here, we need to care about the underlying orientations of spin manifolds. See the arguments in \cite[Subsection 4.2.3]{O}.)
Then, the existence of a DGA map $\Phi_{L_{\varphi}}$ follows from Theorem \ref{thm-DGA-map} and the description of the base change from $R_-$ to $R_+$ follows from Lemma \ref{lem-i-pm}.
\end{proof}

Lastly, let us introduce a spin structure $\mathfrak{s}^0_K$ on $\Lambda_K$ in the manner of \cite[Section 6.2]{EENS}. When $T\Lambda_K$ is trivialized as an oriented vector bundle, then $\mathfrak{s}^0_K$ is given by the trivial $\mathrm{Spin}(2)$-bundle on $\Lambda_K$. It determines the orientations of moduli spaces which are used to define the differential of the DGAs  in \cite[Section 2.3.4]{EENS-filt} and \cite[Section 3.6]{EENS}.
%
The ring isomorphism $\psi(\mathfrak{s}^0_K, \mathfrak{s}_K)$ from (\ref{change-spin}) is given by
\begin{align}\label{isom-s-s0}
\psi(\mathfrak{s}^0_K, \mathfrak{s}_K) \colon \Z[\lambda^{\pm},\mu^{\pm},U^{\pm}] \to  \Z[\lambda^{\pm},\mu^{\pm},U^{\pm}] \colon \lambda \mapsto \lambda,\ \mu \mapsto -\mu, \  U\mapsto U .
\end{align}


\section{Proof of main results}\label{sec-proof}

In this section, we abbreviate the Laurent polynomial ring $\Z[\lambda^{\pm},\mu^{\pm}, U^{\pm}]$ by $R$.

\subsection{General case}\label{subsec-general}

We will combine Proposition \ref{prop-DGA-map-conormal} with the results in \cite{EENS-filt, N-comb, N-intro} about the knot DGA.
Some conventions about the knot DGA in these papers are slightly different from each other, as summarized in \cite[Appendix A]{N-intro}.
In this section, we mainly refer to \cite{N-intro} for the definition of the knot DGA.

Let $K$ be an oriented knot in $\R^3$ given by the closure of a braid $B$.
Then, a DGA over $R$, called the knot DGA of $K$, is constructed from $B$ as in \cite[Definition 3.11]{N-intro}.
Its stable tame isomorphism class as a DGA over $R$ depends only on the isotopy class of $K$ as an oriented knot.
See \cite[Theorem 1.4]{N-comb} and \cite[Theorem 3.1]{N-intro}.
Let us denote this DGA by $(\mathcal{A}^{\knot}_*(K),\partial_K)$ and its homology $\ker \partial_K / \Im \partial_K$ by $HC_*(K)$.
We will review a description of the $0$-th degree part $HC_0(K)$ in Section \ref{subsec-thm1}.

\begin{thm}[Theorem 1.2 of \cite{EENS-filt}, Theorem 3.14 and Appendix A of \cite{N-intro} ]\label{thm-KCH}
$(\mathcal{A}_*(\Lambda_K), \partial_{J,\mathfrak{s}^0_K})$
 is stable tame isomorphic over $R$ to $(\mathcal{A}^{\knot}_*(K),\partial_K)$ whose $R$-algebra structure is changed via a ring isomorphism
\[\psi' \colon R\to R \colon \lambda\mapsto -\lambda,\ \mu\mapsto -\mu,\ U\mapsto U.\]
\end{thm}
The base change by $\psi'$ is due to the difference of conventions between \cite{EENS-filt} and \cite{N-intro}. See \cite[Appendix A]{N-intro}.
It follows from this theorem that there are unital graded ring homomorphisms
\begin{align}\label{knot-conormal}
\begin{array}{cc} I_K \colon  \mathcal{A}_*(\Lambda_K) \to \mathcal{A}^{\knot}_*(K) ,  &
\bar{I}_K \colon \mathcal{A}^{\knot}_*(K) \to \mathcal{A}_*(\Lambda_K) ,
 \end{array}\end{align}
intertwining $\partial_K$ and $\partial_{J,\mathfrak{s}^0_K}$ such that
\[\begin{array}{cc}
 I_K(r\cdot x) = \psi'(r) \cdot I_K(x), &
 \bar{I}_K (r\cdot y) = (\psi')^{-1}(r) \cdot \bar{I}_K(y),
 \end{array}\]
for every $r\in R$, $x\in  \mathcal{A}_*(\Lambda_K)$ and $y \in \mathcal{A}^{\knot}_*(K) $.

We continue to consider the setup in Subsection \ref{subsubsec-clean}. Let $K_0$ and $K_1$ be oriented knots in $\R^3$. For $\varphi \in \Ham_c(T^*\R^3)$, suppose that $\varphi(L_{K_0})$ and $\R^3$ have a clean intersection along $K_1$.
Then we have the following theorem analogous to \cite[Proposition 4.6]{O}.
\begin{thm}\label{thm-map-HC}
There exist integers $a,b,c$ and $d$ with $a\in \{\pm 1\}$ and a unital graded ring homomorphism
\[\Psi_{\varphi} \colon \mathcal{A}^{\knot}_*(K_0) \to \mathcal{A}^{\knot}_*(K_1)\]
such that $\Psi_{\varphi} \circ \partial_{K_0} = \partial_{K_1}\circ \Psi_{\varphi}$ and
\begin{align}\label{commute-R}
\Psi_{\varphi}(r\cdot x) = \psi(r)\cdot \Psi_{\varphi}(x)
\end{align}
for every $x \in \mathcal{A}^{\knot}_*(K_0)$ and $r\in R$, where
$\psi$ is a ring isomorphism defined by
\[ \psi \colon R \to R \colon \lambda\mapsto \lambda^a\mu^bU^c,\ \mu \mapsto \mu^a U^d,\ U \mapsto U.\]
\end{thm}
\begin{proof}
A unital graded ring homomorphism $\Psi_{\varphi}$ satisfying $\Psi_{\varphi} \circ \partial_{K_0} = \partial_{K_1}\circ \Psi_{\varphi}$ is defined by the composition of the following morphisms of differential graded rings:
\[
\xymatrix@C=40pt{
(\mathcal{A}^{\knot}_*(K_0) ,\partial_{K_0}) \ar[r]^-{\bar{I}_{K_0}} & (\mathcal{A}_*(\Lambda_{K_0}) ,\partial_{J, \mathfrak{s}^0_{K_0}}) \ar[r]_-{\cong}^-{I(\mathfrak{s}^0_{K_0}, \mathfrak{s}_{K_0})} & (\mathcal{A}_*(\Lambda_{K_0}) ,\partial_{J, \mathfrak{s}_{K_0}}) \ar[d]^-{\Phi_{L_{\varphi}}} \\
(\mathcal{A}^{\knot}_*(K_1) ,\partial_{K_1}) & (\mathcal{A}_*(\Lambda_{K_1}) ,\partial_{J, \mathfrak{s}^0_{K_1}}) \ar[l]_-{I_{K_1}} & (\mathcal{A}_*(\Lambda_{K_1}) ,\partial_{J, \mathfrak{s}_{K_1}}) .\ar[l]^-{\cong}_-{I(\mathfrak{s}^0_{K_1}, \mathfrak{s}_{K_1})^{-1}} 
}
\]
Here, $\bar{I}_{K_0}$ and $I_{K_1}$ are the morphisms (\ref{knot-conormal}) from Theorem \ref{thm-KCH}.
$I(\mathfrak{s}^0_{K_i}, \mathfrak{s}_{K_i})$ for $i\in \{0,1\}$ are the isomorphisms in Proposition \ref{prop-spin}. The vertical morphism $\Phi_{L_{\varphi}}$ is given by Proposition \ref{prop-DGA-map-conormal}.
From the relation between each morphism and the $R$-algebra structure on each DGA, 
the relation (\ref{commute-R}) holds if we define $\psi$ by
\[\psi \coloneqq \psi' \circ \psi(\mathfrak{s}^0_{K_1},\mathfrak{s}_{K_1})^{-1}  \circ  (\ref{isom-i-pm}) \circ  \psi(\mathfrak{s}^0_{K_0},\mathfrak{s}_{K_0}) \circ (\psi')^{-1}  ,\] 
where $\psi'$ is defined in Theorem \ref{thm-KCH} and $\psi(\mathfrak{s}^0_{K_i},\mathfrak{s}_{K_i})$ for $i\in \{0,1\}$ are given by (\ref{isom-s-s0}).
We compute that for $i\in \{0,1\}$
\[\psi(\mathfrak{s}^0_{K_i},\mathfrak{s}_{K_i}) \circ (\psi')^{-1} \colon \lambda \mapsto -\lambda,\ \mu\mapsto \mu,\ U\mapsto U,\]
then $\psi$ is described as in the theorem
by using the integers $a,b,c$ and $d$ which determine (\ref{isom-i-pm}).
\end{proof}

We recall the definition of augmentations of a DGA. See, for instance, \cite[Definition 2.6]{N-intro}.
\begin{defi}\label{def-aug}
Let $\bold{k}$ be a field and let $(\mathcal{A}_*,\partial)$ be a unital non-commutative DGA over $R$. A ring homomorphism $\epsilon \colon \mathcal{A}_* \to \bold{k}$ is called an \textit{augmentation} of  $(\mathcal{A}_*,\partial)$ to $\bold{k}$ if $\epsilon (\bold{1}) =1$, $\epsilon(x) =0$ for any $x\in \mathcal{A}_p$ with $p\neq 0$, and $\epsilon \circ \partial \colon \mathcal{A}_1 \to \bold{k}$ is the zero map.
Here, $\bold{1}$ is the unit of $\mathcal{A}_*$.
\end{defi}

In the following, given an augmentation $\epsilon$ as in the definition, we abbreviate $\epsilon(r\cdot \bold{1})$ by $\epsilon(r)$ for any $r\in R$.

Let us also consider a knot invariant defined from $\bold{k}$-valued augmentations of the knot DGA.
\begin{defi}[Definition 5.1 of \cite{N-intro}]\label{def-aug-var}
Fix a field $\bold{k}$.
For an oriented knot $K$ in $\R^3$, its \textit{augmentation variety} over $\bold{k}$ is a subset of $(\bold{k}^*)^3$ defined by
\[ V_{\bold{k}}(K) \coloneqq \left\{ (x,y,Z) \in (\bold{k}^*)^3\  \middle| \begin{array}{l}  \text{there exists an augmentation }\epsilon \text{ of } (\mathcal{A}^{\knot}_*(K),\partial_K) \text{ to } \bold{k}  \\
\text{such that } \epsilon(\lambda)=x,\ \epsilon (\mu )= y \text{ and }\epsilon(U )=Z
\end{array}
\right\}. \]
\end{defi}

From Theorem \ref{thm-map-HC}, we obtain a refinement of \cite[Theorem 1.1]{O}.
\begin{cor}\label{cor-aug}
There exist $a\in \{\pm 1\}$ and integers $b,c$ and $d$ for which we have an injective map
\[ V_{\bold{k}}(K_1) \hookrightarrow V_{\bold{k}}(K_0) \colon (x,y,Z) \mapsto (x^ay^{b}Z^{c}, y^a Z^{d} ,Z).\]
\end{cor}
\begin{proof}
Using $\Psi_{\varphi}$ of Theorem \ref{thm-map-HC}, it follows that
for any augmentation $\epsilon$ of $(\mathcal{A}_*^{\knot}(K_1),\partial_{K_1})$ to $\bold{k}$ such that $(\epsilon(\lambda), \epsilon(\mu ) ,\epsilon(U ) )= (x,y,Z)$,
the composite map $\epsilon \circ \Psi_{\varphi} \colon \mathcal{A}_*^{\knot}(K_0) \to \bold{k}$ is an augmentation of $(\mathcal{A}_*^{\knot}(K_0),\partial_{K_0})$ to $\bold{k}$ such that
\[(\epsilon \circ \Psi_{\varphi} (\lambda), \epsilon \circ \Psi_{\varphi} (\mu), \epsilon \circ \Psi_{\varphi} (U ))= (x^ay^bZ^c, y^aZ^d,Z). \]
From the definition of $V_{\bold{k}}(K_i)$ for $i\in \{0,1\}$, we obtain a map from $V_{\bold{k}}(K_1)$ to $V_{\bold{k}}(K_0)$ as in the assertion.
\end{proof}


The next proposition is used to prove Theorem \ref{thm-main} and Theorem \ref{thm-eight} by taking $\bold{k}=\Q$.
Consider a projection $\pi_{y,Z}\colon (\bfk^*)^3 \to (\bfk^*)^2\colon (x,y,Z) \mapsto (y,Z)$.

\begin{prop}\label{prop-unknot-aug}
Suppose that $K_1$ is the unknot. Then,
\[  \{ (y,Z) \in (\bfk^*)^2 \mid y\neq Z^n \text{ for all }n\in \Z \} \subset \pi_{y,Z} ( V_{\bfk}(K_0)). \]
\end{prop}
\begin{proof}
If $K_1$ is the unknot,
from the trivial braid on a single strand, we can compute as in \cite[Example 3.13]{N-intro} that $\mathcal{A}^{\knot}_0(K_1) = R=\Z[\lambda^{\pm},\mu^{\pm},U^{\pm}]$ and $\partial_{K_1}(\mathcal{A}_1^{\knot}(K_1))$ is the ideal generated by
\[ U-\lambda-\mu+\lambda \mu .\]
Let us take $a\in \{\pm 1\}$ and $b,c,d\in \Z$ to be the integers in Corollary \ref{cor-aug}.
For any $y,Z\in \bold{k}^*$ satisfying $y \notin \{Z^n \mid n\in \Z \}$, we take an augmentation $\epsilon_1 \colon \mathcal{A}^{\knot}_*(K_1) \to \bold{k}$ determined by
\[\begin{array}{ccc} \epsilon_1 (\lambda) =  \displaystyle{ \frac{Z - (y\cdot Z^{-d})^a}{1- (y\cdot Z^{-d})^a} } , & \epsilon_1 (\mu)= (y\cdot Z^{-d})^a, & \epsilon_1 (U) =Z .\end{array}\]
Note that
$ y\cdot Z^{-d} \neq 1 $ and $y\cdot Z^{-d} \neq Z^a$ from the condition on $y$ and $Z$.
By the map of Corollary \ref{cor-aug}, $(\epsilon_1(\lambda), (y\cdot Z^{-d})^a, Z)\in V_{\bfk}(K_1)$ is sent to an element of $V_{\bfk}(K_0)$ whose $(y,Z)$-component is
$ ( (y\cdot Z^{-d})  \cdot Z^d    , Z ) = (y,Z)$.
\end{proof}

\subsection{Proof of Theorem \ref{thm-1} for $K'\# T_{(2,2m+1)}$ ($m\neq 0,-1$)}\label{subsec-thm1}

\subsubsection{Review of the definition of $HC_0(K)$}\label{subsubsec-HC0}
First, let us review how to compute $HC_0(K)$ for an oriented knot $K$. We refer to \cite[Section 3]{N-intro}.
Note that since $\mathcal{A}^{\knot}_p(K)=0$ for $p<0$, $HC_0(K) = \mathcal{A}^{\knot}_0(K)/ \partial_K(\mathcal{A}^{\knot}_1(K)) $.

For any $n\in \Z_{\geq 1}$, let $B_n$ denote the group of braids on $n$ strands, which is generated by $\{\sigma_1,\dots ,\sigma_{n-1}\}$ with relations $\sigma_i\sigma_{i+1}\sigma_i = \sigma_{i+1}\sigma_i \sigma_{i+1}$ for   $1 \leq i \leq n-2$ and  $\sigma_i\sigma_j = \sigma_j\sigma_i$ for $|i - j| \geq 2$.
Let $\mathcal{A}_n$ be the $R$-algebra freely generated by the set $\{a_{i,j} \mid i,j \in \{1,\dots ,n\},\ i\neq j\}$.
Then, a group homomorphism $B_{n} \to \mathrm{Aut} (\mathcal{A}_{n}) \colon B \mapsto \phi_{B}$ is defined as in \cite[Definition 3.3]{N-intro}.
It is determined by $\phi_{\sigma_k}\colon \mathcal{A}_n\to \mathcal{A}_n$ for each generator $\sigma_k$, and $\phi_{\sigma_k}$ is given by
\begin{align}\label{phi-sigma}
\phi_{\sigma_k}\colon
\begin{cases} 
a_{i,j} \mapsto a_{i,j} & \text{ if }i,j \neq k,k+1, \\
a_{k+1,i} \mapsto a_{k,i} & \text{ if } i \neq k,k+1 , \\
a_{i,k+1} \mapsto a_{i,k} & \text{ if } i\neq k,k+1, \\
a_{k,k+1} \mapsto -a_{k+1,k} , &  \\
a_{k+1,k} \mapsto -a_{k,k+1}, & \\
a_{k,i} \mapsto a_{k+1,i}- a_{k+1,k}a_{k,i} & \text{ if }i\neq k,k+1 , \\
a_{i,k} \mapsto a_{i,k+1}-a_{i,k}a_{k,k+1} & \text{ if }i\neq k,k+1 .
\end{cases}
\end{align}
See also \cite[Section 2.1]{EENS}. We remark that $\phi_{B\cdot B'} = \phi_{B}\circ \phi_{B'}$ for any $B,B'\in B_n$.

The automorphism $\phi_B\in \mathrm{Aut}( \mathcal{A}_{n})$ for $B\in B_n$ is extended to an automorphism on $\mathcal{A}_{n+1}$ as follows:
We consider the group homomorphism $B_{n+1} \to \mathrm{Aut}(\mathcal{A}_{n+1})$ defined as above and an injective map $B_n \to B_{n+1}$ defined by adding the extra $(n+1)$-th strand to any braid on $n$ strands (see \cite[Remark 3.5]{N-intro}).
Then, for any $B\in B_n$, $\phi_B \in \mathrm{Aut} (\mathcal{A}_{n+1})$ is defined via the composite map $B_n\hookrightarrow B_{n+1} \to \mathrm{Aut} (\mathcal{A}_{n+1})$.
In addition, let us label the $(n+1)$-th strand by $*$ and denote $a_{i,n+1}$ by $a_{i,*}$ and $a_{n+1,i}$ by $a_{*,i}$ for every $i\in \{1,\dots ,n\}$.

We define $n\times n$ matrices $\bold{A}$ and $\widehat{\bold{A}}$ by
\[\begin{array}{ll}
\bold{A}_{i,j} \coloneqq \begin{cases} a_{i,j} & \text{ if }i<j , \\ 1-\mu & \text{ if }i=j , \\ -\mu a_{i,j} & \text{ if }i>j , \end{cases}  & 
\widehat{\bold{A}}_{i,j} \coloneqq \begin{cases} U a_{i,j} & \text{ if }i<j , \\ U-\mu & \text{ if }i=j , \\ -\mu a_{i,j} & \text{ if }i>j . \end{cases}
\end{array}\]

Fix $B\in B_n$ which has the writhe $w$ (the sum of the exponents in the braid
word).
Let $\boldsymbol{\Lambda}$ be an $n\times n$ diagonal matrix such that $\boldsymbol{\Lambda}_{1,1}= \lambda \mu^{w}U^{-(w-n+1)/2}$ and $\boldsymbol{\Lambda}_{i,i} =1$ for every $i\in \{2,\dots ,n\}$.
Let $\boldsymbol{\Phi}^L_{B}$ be an $n\times n$ matrix whose entries $ (\boldsymbol{\Phi}^L_{B})_{i,j} \in \mathcal{A}_n$ for $i,j\in \{1,\dots ,n\}$ are determined by
\[ \phi_{B}(a_{i,*}) = \sum_{j=1}^n (\boldsymbol{\Phi}^L_{B})_{i,j}a_{j,*}. \]

Let $K$ be the oriented knot in $\R^3$ given by the closure of the braid $B$.
By \cite[Theorem 4.1, Remark 4.2]{N-intro}, $HC_0(K)$ is the quotient of $\mathcal{A}_n$ by the two-sided ideal generated by all entries of the two matrices
\begin{align}\label{ideal-generator}
\begin{array}{cc}
\bold{A} - \boldsymbol{\Lambda} \cdot \phi_B (\bold{A}) \cdot \boldsymbol{\Lambda}^{-1}, &
\widehat{\bold{A}} - \boldsymbol{\Lambda} \cdot \boldsymbol{\Phi}^L_{B} \cdot \bold{A},
\end{array}
\end{align}
where $(\phi_B (\bold{A}))_{i,j} \coloneqq \phi_{B} (\bold{A}_{i,j})$ for $i,j\in \{1,\dots ,n\}$. The isomorphism class of $HC_0(K)$ as a unital $R$-algebra depends only on the isotopy class of $K$ as an oriented knot. See \cite[Theorem 3.15]{N-intro} and \cite[Theorem 1.4 (2)]{N-comb}.

We consider a commutative $R$-algebra $HC_0^{\ab}(K)$ defined as the quotient of $HC_0(K)$ by the relations $\xi \eta = \eta \xi$ for all $\xi,\eta \in HC_0(K)$.
Using the $R$-algebra $R[a_{i,j}\mid 1\leq i,j\leq n,\ i\neq j]$ of polynomials  and a natural map
\begin{align}\label{map-pi}
\pi_n\colon \mathcal{A}_n\to R[a_{i,j}\mid 1\leq i,j\leq n,\ i\neq j] \colon a_{i,j} \mapsto a_{i,j},
\end{align}
we have
\[ HC_0^{\ab}(K) = R[a_{i,j}\mid 1\leq i,j\leq n,\ i\neq j] / \mathcal{I}_{B},\]
where $\mathcal{I}_{B}$ is the ideal generated by the images under $\pi_n$ of all entries of the two matrices (\ref{ideal-generator}).

\subsubsection{On augmentations of $\mathcal{A}^{\knot}_*(K\# T_{(2,2m+1)})$ for $m\geq 0$}\label{subsubsec-torus}

Fix any $B\in B_n$ whose closure gives a knot $K$ in $\R^3$.
We consider a braid $B^{\#}_m$ on $(n+2)$ strands defined by 
\[ B^{\#}_m \coloneqq  \sigma_n \cdot  (\sigma_{n+1})^{2m+1} \cdot B \in B_{n+2} \]
for every $m\in \Z_{\geq 0}$.
Then, the closure of $B^{\#}_m$ gives the connected sum $K\# T_{(2,2m+1)}$ of $K$ and the $(2,2m+1)$-torus knot.
(When $m=0$, $T_{(2,1)}$ is the unknot.)


By using  (\ref{phi-sigma}), let us compute
$\phi_{B^{\#}_m}(a_{n+2,*})$, which is equal to
\[\left( \phi_{\sigma_n}\circ \phi_{(\sigma_{n+1})^{2m+1}} \circ \phi_B \right) (a_{n+2,*}) .\]
First, since $\phi_{\sigma_k}(a_{n+2,*}) = a_{n+2,*}$ for every $k\in \{1,\dots ,n-1\}$, $\phi_B(a_{n+2,*}) = a_{n+2,*}$. Next, we compute $\phi_{(\sigma_{n+1})^{2m+1}}(a_{n+2,*})$ by induction on $m=0,1,\dots $.

\begin{lem}
Let $\hat{f}_m$ and $\hat{g}_m$ be elements of the free ring $ \Z\la a_{n+1,n+2}, a_{n+2,n+1} \ra$ determined inductively on $m=0,1,\dots$ by $\hat{f}_0=1$, $\hat{g}_0=0$, and
\begin{align}\label{hat-fg-induction}
 \begin{cases} \hat{f}_{m+1} = \hat{f}_m \cdot (1- a_{n+1,n+2}a_{n+2,n+1}) -\hat{g}_m \cdot a_{n+2,n+1} ,\\
 \hat{g}_{m+1} =  \hat{f}_m \cdot a_{n+1,n+2} +\hat{g}_m . \end{cases}
 \end{align}
 Then, for every $m\in \Z_{\geq 0}$,
 \begin{align}\label{hat-fg}
\phi_{(\sigma_{n+1})^{2m+1}} (a_{n+2,*}) = \hat{f}_m \cdot a_{n+1,*} + \hat{g}_m\cdot a_{n+2,*}. 
\end{align}
\end{lem}
\begin{proof}
We prove by induction on $m$.
Since $\phi_{\sigma_{n+1}}(a_{n+2,*}) = a_{n+1,*}$, (\ref{hat-fg}) for $m=0$ is satisfied by setting $\hat{f}_0 = 1$ and $\hat{g}_0=0$.
Suppose that (\ref{hat-fg}) holds for some $m\in \Z_{\geq 0}$.
For readability, let us rewrite
\[\begin{array}{cccc}
a_{n+1,n+2}=p, & a_{n+2,n+1}=q, & \hat{f}_m = \hat{f}_m(p,q), & \hat{g}_m=\hat{g}_m(p,q).
\end{array}\] 
Note that $\phi_{\sigma_{n+1}}$ maps $p$ to $-q$ and $q$ to $-p$.
Let $\hat{f}_m(-q,-p) , \hat{g}_m(-q,-p) \in \Z\la p,q\ra$ be the elements obtained from $\hat{f}_m(p,q)$ and $\hat{g}_m(p,q)$ by replacing $p$ with $-q$ and $q$ with $-p$. Then, we can compute that
\begin{align*}
 \phi_{(\sigma_{n+1})^{2m+2}} (a_{n+2},*) 
 & = \phi_{\sigma_{n+1}} \left( \hat{f}_m(p,q) \cdot a_{n+1,*} + \hat{g}_m(p,q)\cdot a_{n+2,*}\right) \\
& = \hat{f}_m(-q,-p) \cdot (a_{n+2,*}-q\cdot a_{n+1,*}) + \hat{g}_m(-q,-p)\cdot a_{n+1,*} \\
& = \left( -\hat{f}_m(-q,-p) q +  \hat{g}_m(-q,-p) \right) \cdot a_{n+1,*} +  \hat{f}_m(-q,-p) \cdot a_{n+2,*}, \\
 \phi_{(\sigma_{n+1})^{2m+3}} (a_{n+2},*) 
& = \left( -\hat{f}_m(p,q)(-p) +\hat{g}_m(p,q) \right) \cdot (a_{n+2,*}-q\cdot a_{n+1,*}) + \hat{f}_m(p,q) \cdot a_{n+1,*} \\
&=\left( \hat{f}_m(p,q) (1 -pq) -\hat{g}_m(p,q)q \right)\cdot a_{n+1,*} + \left( \hat{f}_m(p,q) p +\hat{g}_m(p,q) \right)\cdot a_{n+2,*}.
\end{align*}
Therefore, if we set $\hat{f}_{m+1}$ and $\hat{g}_{m+1}$ by (\ref{hat-fg-induction}),
 then (\ref{hat-fg}) with $m$ replaced by $m+1$ holds.
 \end{proof}
We apply $\phi_{\sigma_n}$ to (\ref{hat-fg}), then
\begin{align*}
 \phi_{B^{\#}_m}(a_{n+2,*}) & = \left( \phi_{\sigma_n}\circ \phi_{(\sigma_{n+1})^{2m+1}} \circ \phi_B \right) (a_{n+2,*}) \\ 
 & =\left( \phi_{\sigma_n}\circ \phi_{(\sigma_{n+1})^{2m+1}} \right) (a_{n+2,*}) \\
&= \phi_{\sigma_n}(\hat{f}_m) \cdot a_{n,*} + \phi_{\sigma_n}(\hat{g}_m)\cdot a_{n+2,*}. 
\end{align*}
Since $\phi_{\sigma_n}$ maps $a_{n+1,n+2}$ to $a_{n,n+2}$ and $a_{n+2,n+1}$ to $a_{n+2,n}$,
the elements $\phi_{\sigma_n}(\hat{f}_m), \phi_{\sigma_n}(\hat{g}_m) \in \Z\la a_{n,n+2}, a_{n+2,n} \ra$ are obtained from $\hat{f}_m$ and $\hat{g}_m$ by replacing $a_{n+1,n+2}$ with $a_{n,n+2}$ and $a_{n+2,n+1}$ with $a_{n+2,n}$. 

It follows that the bottom ($n+2$)-th row of $\boldsymbol{\Phi}^L_{B^{\#}_m}$ is
\[ \begin{pmatrix} 0 & \cdots & 0 & \phi_{\sigma_n}(\hat{f}_m) & 0 & \phi_{\sigma_n}(\hat{g}_m)   \end{pmatrix}.\]
Let $c^m_{i,j} \in \mathcal{A}_{n+2}$ denote the $(i,j)$-th entry of  $\widehat{\bold{A}} - \boldsymbol{\Lambda} \cdot \boldsymbol{\Phi}^{L}_{B^{\#}_m} \cdot \bold{A}$.
We focus on the $(n+2,n)$-th and $(n+2,n+2)$-th entries:
\begin{align}\label{cm-ij}
\begin{split}
c^m_{n+2,n} & =  \widehat{\bold{A}}_{n+2,n} - \left( \phi_{\sigma_n}(\hat{f}_m) \cdot  \bold{A}_{n,n} + \phi_{\sigma_n}(\hat{g}_m) \cdot \bold{A}_{n+2,n} \right) \\
& = - \mu a_{n+2,n} - \left( \phi_{\sigma_n}(\hat{f}_m) \cdot (1-\mu) + \phi_{\sigma_n}(\hat{g}_m) \cdot (-\mu a_{n+2,n}) \right), \\
c^m_{n+2,n+2} & =  \widehat{\bold{A}}_{n+2,n+2} - \left( \phi_{\sigma_n}(\hat{f}_m)\cdot  \bold{A}_{n,n+2} + \phi_{\sigma_n}(\hat{g}_m) \cdot \bold{A}_{n+2,n+2} \right) \\
&= (U-\mu) - \left( \phi_{\sigma_n}(\hat{f}_m) \cdot  a_{n,n+2} + \phi_{\sigma_n}(\hat{g}_m) \cdot (1-\mu)\right) .
\end{split}
\end{align}
Using the map $\pi_{n+2}$ of (\ref{map-pi}), let us define elements of $R[a_{i,j} \mid 1\leq i,j \leq n+2,\ i\neq j]$ by
 \[\begin{array}{cc}
 f_m \coloneqq \pi_{n+2} ( \phi_{\sigma_n}(\hat{f}_m)), & g_m \coloneqq \pi_{n+2} ( \phi_{\sigma_n}(\hat{g}_m)) , \end{array}\]
and denote $X = a_{n,n+2}$ and $Y = a_{n+2,n}$.
By (\ref{hat-fg-induction}), $f_m$ and $g_m$ are polynomials in $\Z[X,Y]$ determined inductively by $f_0= 1$, $g_0=0$ and
\begin{align}\label{induction-fmgm}
\begin{array}{cc}
f_{m+1} = (1-XY) f_m - Yg_m , & g_{m+1} = X f_m + g_m,
\end{array}
\end{align}
for every $m\in \Z_{\geq 0}$.
Let us also define $F_m,G_m\in \Z[\mu] [X,Y]$ by
\[ \begin{array}{cc} 
F_m \coloneqq (1-\mu) f_m - \mu Y g_m, & G_m \coloneqq  X f_m+(1-\mu)g_m .
\end{array}\]
Then, $\pi_{n+2}(c^m_{n+2,n}) = -\mu Y - F_m$ and $\pi_{n+2} ( c^m_{n+2,n+2}) = U-\mu - G_m$ by (\ref{cm-ij}).

\begin{lem}\label{lem-FG}
Let us define  $h_m,k_m\in \Z[\mu][T]$ inductively on $m=0,1,\dots$ by $h_0=1-\mu$, $k_0= 1$ and
\begin{align}\label{induction-hk}
 \begin{array}{cc} h_{m+1} = h_m - T \cdot  k_m , & k_{m+1} = h_m + (1-T)\cdot k_m, \end{array}
 \end{align}
Then, for every $m\in \Z_{\geq 0}$,
\begin{align}\label{hk}
\begin{array}{cc} F_m (X,Y) = h_m(XY), & G_m(X,Y) = X\cdot k_m(XY) .
\end{array}
\end{align}

\end{lem}
\begin{proof}We prove by induction on $m$.
Since $F_0 =1-\mu$ and $G_0 =X$, (\ref{hk}) for $m=0$ is satisfied by setting $h_0\coloneqq 1-\mu$ and $k_0\coloneqq 1$.
Suppose that (\ref{hk}) holds for some $m \in \Z_{\geq 0}$.
Using (\ref{induction-fmgm}), we compute that
\begin{align*}
F_{m+1}  & =  (1-\mu) f_{m+1} -\mu Y g_{m+1} \\
& =  (1-\mu) ((1-XY)f_m-Yg_m) -\mu Y(Xf_m +g_m) \\
& = (1-\mu-XY)f_m - Yg_m \\
& = ((1 -\mu)f_m -\mu Y g_m) - Y( X f_m+(1-\mu)g_m ) \\
& = F_m - Y G_m \\
& = h_m(XY) - XY \cdot k_m(XY), \\
G_{m+1} & = Xf_{m+1} + (1-\mu)g_{m+1} \\
& = X ((1-XY)f_m-Yg_m) + (1-\mu) (Xf_m +g_m)  \\
& = X(2-\mu-XY) f_m + (1-\mu-XY) g_m \\
& = X ((1-\mu) f_m -\mu Y g_m) + (1-XY)(X f_m +(1-\mu) g_m) \\
& = XF_m +(1-XY)G_m \\
& = X \left( h_m(XY) +(1-XY)\cdot k_m(XY) \right) .
\end{align*}
Therefore, if we set $h_{m+1},k_{m+1} \in \Z[\mu][T]$ by (\ref{induction-hk}),
then (\ref{hk}) with $m$ replaced by $m+1$ holds.
\end{proof}

For every $m\in \Z_{\geq 0}$, let us define a polynomial $P_m \in \Z[\mu,U][T] \subset R[T]$ by
\begin{align}\label{poly-Pm}
 P_m \coloneqq (U-\mu) \cdot h_m + \mu T \cdot k_m . 
 \end{align}
 
%
\begin{prop}\label{prop-Jm}
For any field $\bold{k}$,
\[  \pi_{y,Z}(V_{\bfk}(K\# T_{(2,2m+1)})) \subset \{ (y,Z)\in (\bfk^*)^2 \mid \rest{P_m}{\mu = y, U =Z}  \text{ has a root in }\bfk  \}. \]
\end{prop}
\begin{proof}
By definition, the ideal $\mathcal{I}_{B^{\#}_m}$ contains $\pi_{n+2}(c^m_{n+2,n}) = -\mu Y - F_m$ and $ \pi_{n+2} ( c^m_{n+2,n+2}) = U-\mu - G_m$.
Moreover, we have an identity
\begin{align*}
P_m(XY) & =  (U-\mu) \cdot h_m (XY) +  \mu XY \cdot k_m(XY) \\
& = (U-\mu) \left( h_m(XY) + \mu Y \right) - \mu Y\left( U-\mu - X\cdot k_m(XY) \right)  \\
 & = (U-\mu)  (\mu Y +F_m(X,Y)) - \mu Y (U-\mu - G_m(X,Y)),
\end{align*}
which implies that  $P_m(XY) \in \mathcal{I}_{B^{\#}_m}$.
Therefore, there exists a well-defined unital $R$-algebra map
\[J_m\colon R[T] / (P_m ) \to  R[a_{i,j} \mid 1\leq i,j \leq n+2,\ i\neq j]/ \mathcal{I}_{B^{\#}_m} = HC^{\ab}_0(K\# T_{(2,2m+1)}) \]
which maps $T $ to $XY= a_{n,n+2}a_{n+2,n} $.

Given a $\bold{k}$-valued augmentation $\epsilon$ of $\mathcal{A}^{\knot}_*(K\# T_{(2,2m+1)})$, it induces a unital ring homomorphism from  $HC^{\ab}_0(K\# T_{(2,2m+1)})$ to $\bfk$.
By precomposing $J_m$, we obtain a unital ring homomorphism $ R[T] / (P_m )  \to \bold{k}$ which sends $\mu$ to $\epsilon(\mu)$ and $U$ to $\epsilon(U)$. This means that $\rest{P_m}{\mu=\epsilon(\mu),U=\epsilon(U)}\in \bold{k}[T]$ has a root in $\bold{k}$.
Now, the proposition follows from the definition of $V_{\bold{k}}(K\#T_{(2,2m+1)})$. 
\end{proof}

\begin{ex}
Let us check the existence of a map as $J_m$ in a simple case.
We have $h_1 = 1-\mu -T$, $k_1 = (1-\mu) + (1-T) = 2-\mu- T$ and 
\begin{align*}
 P_1 (T) &=  (U-\mu) ( 1-\mu -T) + \mu T(2-\mu-T) \\
 &= -\mu T^2 + ( 3\mu - \mu^2- U ) T +   (U-\mu)(1-\mu) .
 \end{align*}
From a braid representation $(\sigma_1)^3 \in B_2$ of $T_{(2,3)}$, it is not difficult to compute $HC_0(T_{(2,3)})$ as outlined in \cite[Exercise 3.10]{N-intro}.
$HC_0(T_{(2,3)})$ is isomorphic to a commutative $R$-algebra $R[X] / (G_1,G_2)$, where
\[\begin{array}{cc} G_1(X) \coloneqq U X^2 -\mu U X + \lambda \mu^3(1-\mu), & G_2(X) \coloneqq U X^2 + \lambda\mu^2 X + \lambda\mu^2(\mu-U) .\end{array}\]
See \cite[Exercise 5.5]{N-intro}.
We have identities
\begin{align*}
X G_1 + \mu G_2& = UX^3 + \lambda\mu^3( 2-\mu ) X+  \lambda\mu^3(\mu-U), \\ 
 UX (X G_1 + \mu G_2 ) & = (UX^2)^2 + \lambda\mu^3(2-\mu) (U X^2) +  \lambda\mu^3(\mu-U)UX , \\
 \lambda\mu^2(\mu-U) G_1 & = \lambda\mu^2(\mu-U) (UX^2) -\lambda\mu^3(\mu-U) UX + \lambda^2 \mu^5 (1-\mu)(\mu-U) , 
 \end{align*}
and
\begin{align*}
 UX (X G_1 + \mu G_2 )  + \lambda\mu^2(\mu-U) G_1 
=\ &  (UX^2)^2 + \lambda\mu^2 (3\mu -\mu^2 -U) (UX^2) + \lambda^2 \mu^5 (1-\mu)(\mu-U) \\
=\ & - \lambda^2\mu^5\cdot P_1 (- \lambda^{-1} \mu^{-3}U X^2) 
\end{align*}
Therefore, we have a well-defined unital $R$-algebra map 
\[R[T]/(P_1) \to HC^{\mathrm{ab}}_0(T_{(2,3)}) \cong R[X]/(G_1,G_2) \colon T \mapsto - \lambda^{-1} \mu^{-3} UX^2. \]
\end{ex}

\begin{rem}\label{rem-reduction}
Consider the reduction of the coefficient ring from $R$ to $\Z[\lambda^{\pm},\mu^{\pm}]$ by setting $U=1$.
In this case, we obtain $V_{\bfk} \cap \{Z=1\} \subset \{ (x,y) \mid \rest{P_m}{\mu=y,U=1} \text{ has a root in }\bfk \}$.
However, this is a trivial inclusion for the following reason:
By the substitution $U=1$, we have
$\rest{P_0}{U=1} = (1-\mu)^2 + \mu T$, and if we define $h'_m,k'_m \in \Z[\mu^{\pm}]$ by $h'_m\coloneqq \rest{h_m}{T= -\mu^{-1}(1-\mu)^2}$ and $k'_m \coloneqq \rest{k_m}{T= -\mu^{-1}(1-\mu)^2}$, then
\begin{align*}
 \rest{P_{m+1}}{U=1, T =-\mu^{-1}(1-\mu)^2} & =  (1-\mu) \cdot h'_{m+1} + \mu (-\mu^{-1}(1-\mu)^2)\cdot k'_{m+1} \\
 & = (1-\mu) (h'_m + \mu^{-1}(1-\mu)^2k'_m ) - (1-\mu)^2 (h'_m + (1 + \mu^{-1}(1-\mu)^2)k'_m) \\
 & = \mu \left( (1-\mu)h'_m ) - (1-\mu)^2 k'_m \right) \\
 & =\mu \rest{P_{m}}{U=1, T =-\mu^{-1}(1-\mu)^2}.
 \end{align*}
Here, the second equality follows from (\ref{induction-hk}).
By induction on $m=0,1,\dots$, we can see that the polynomial $\rest{P_m}{U=1}\in \Z[\mu^{\pm}][T]$ has a root $T = -\mu^{-1}(1-\mu)^2$.
\end{rem}

\subsubsection{Application of $\Q$-valued augmentations}

\noindent
\textbf{Notation.}
For every $q \in \C (T)$, we define its degree $\deg q\coloneqq \max \{ \deg f , \deg g\} $, where $f,g \in \C[T]$ are coprime polynomials such that $g\neq 0$ and $q=\frac{f}{g}$.
\begin{rem}\label{rem-proj}
Let us take $q=\frac{f}{g} \in \C(T)$ as above.
Then, $\deg q \in \Z_{\geq 0}$ coincides with the degree of the continuous map $\hat{q}\colon \C P^1 \to \C P^1$ such that $\hat{q}([z:1]) = [f(z):g(z)]$ for every $z\in \C$ and
\[ \hat{q} ( [1:0]) = \begin{cases} [\lim_{z\to \infty} \frac{f(z)}{g(z)}  :1]& \text{ if } \deg f \leq \deg g , \\
[1: 0] & \text{ if } \deg f  > \deg g.  \end{cases}\]
Consider $r = \frac{f'}{g'}\in \C(T')$ for $f',g' \in \C(T')$ such that $\rest{g'}{T'= q} \neq 0$ in $\C(T)$. Then, the degree of $\rest{r}{T'= q} \in \C(T)$ is computed by $\deg r\cdot \deg q$, which is the degree of the composite map $\hat{r}\circ \hat{q} \colon \C P^1 \to \C P^1$.
\end{rem}


In the following, we fix $y_0 \in \Q^*\setminus \{1\}$ (one may choose, for instance, $y_0=2$).
By the substitution $\mu= y_0$, let us define $\bar{h}_m,\bar{k}_m \in \Q[T]$ as
\[\begin{array}{cc}
\bar{h}_m\coloneqq \rest{h_m}{\mu=y_0} , & \bar{k}_m\coloneqq \rest{k_m}{\mu=y_0} .
\end{array}\]

\begin{lem}\label{lem-deg}
For every $m\in \Z_{\geq 1}$, $\deg \bar{h}_m =\deg \bar{k}_m =m$, and $\bar{h}_m$ and $T\cdot \bar{k}_m$ are coprime polynomials in $\C[T]$.
In particular, if we define
\[ r_m \coloneqq \frac{T\cdot \bar{k}_m}{\bar{h}_m} \in \C(T), \]
then $\deg r_{m} = \deg (T\cdot \bar{k}_m) =m+1$ for every $m\in \Z_{\geq 1}$.
\end{lem}
\begin{proof}
By Lemma \ref{lem-FG}, we have $\bar{h}_0= 1-y_0$, $\bar{k}_{0}= 1$ and
\[ \begin{array}{cc}  \bar{h}_{m+1} = \bar{h}_{m} - T\cdot \bar{k}_{m}, & \bar{k}_{m+1} = \bar{h}_{m} + (1-T)\cdot \bar{k}_{m} , \end{array}\]
for every $m\in \Z_{\geq 0}$.
Then, it is easy to see that:
\begin{itemize}
\item For every $m\in \Z_{\geq 1}$, both $\bar{h}_{m}$ and $\bar{k}_{m}$ have  $(-1)^m \cdot  T^m$ as the term with the highest degree.
\item For every $m\in \Z_{\geq 0}$, $\bar{h}_m = (1-T) \cdot \bar{h}_{m+1} + T\cdot \bar{k}_{m+1}$ and $T\cdot \bar{k}_m = T\cdot (\bar{k}_{m+1}-\bar{h}_{m+1})$.
\end{itemize}
From the first assertion, $\deg \bar{k}_m=\deg \bar{h}_m =m$ for every $m\in \Z_{\geq 1}$
From the second assertion, if we suppose that $\bar{h}_m$ and $T\cdot \bar{k}_m$ have a common root $T=c$ in $\C$ for some $m\in \Z_{\geq 1}$, then $\bar{h}_{m'}(c) = c \cdot \bar{k}_{m'}(c) =0$ for every $m'\in \{0,\dots ,m\}$, which is a contradiction when $m'=0$ since $1-y_0\neq 0$.
Thus,  $\bar{h}_m$ and $T\cdot \bar{k}_m$ are coprime in $\C[T]$ for every $m\in \Z_{\geq 1}$.
\end{proof}

Now, we prove Theorem \ref{thm-1} for $K'\#T_{(2,q)}$, where $q\neq \pm 1$ is an odd number.
\begin{thm}\label{thm-main}
Let $K_0 = K\# T_{(2,2m+1)}$, where $K$ is an arbitrary knot in $\R^3$ and  $m\in \Z\setminus \{0,-1\}$.
Then, there is no $\varphi \in \Ham_c(T^*\R^3)$ such that $\varphi(L_{K_0})$ and $\R^3$ have a clean intersection along the unknot in $\R^3$.
\end{thm}
\begin{proof}
Let us first consider the case $m\geq 1$.
From  Proposition \ref{prop-unknot-aug} and Proposition \ref{prop-Jm} with $\bold{k}=\Q$, it suffices to show that in $(\Q^*)^2$
\begin{align}\label{non-inclusion} \{ (y,Z)  \mid y\neq Z^n\text{ for all }n\in \Z \} \not\subset \{ (y,Z)  \mid \rest{P_m}{\mu=y,U=Z} \text{ has a root in }\Q \} . 
\end{align}
For the fixed rational number $y_0\neq 1$, note that
\[ \rest{P_m}{\mu=y_0,U=Z} = (Z -y_0)\cdot \bar{h}_{m}(T) + y_0T \cdot \bar{k}_{m}(T)  \]
for any $Z\in \Q^*$.
Let us regard $\rest{P_m}{\mu=y_0,U=Z}$ as a polynomial in $\Q(Z)[T]$. Take $p_1,\dots ,p_e\in \Q(Z)[T]$ to be irreducible elements such that
$  \rest{P_m}{\mu=y_0,U=Z} = p_1\cdots p_e $.

We claim that the degree of $p_i$ as a polynomial in $T$ is greater than or equal to $2$ for every $i\in \{1,\dots ,e\}$.
It is proved by a contradiction as follows: Assume that there exists $i\in \{1,\dots ,e\}$ such that $ p_i = q'(T - q)$ for some $q,q'\in \Q(Z)$ with $q'\neq 0$.
Then, we have an equation
\[( y_0^{-1} Z -1) \cdot \rest{\bar{h}_m}{T=q} + q \cdot \rest{\bar{k}_m}{T=q} =0\]
in $\C(Z)$.
Since $\bar{h}_m$ and $T\cdot \bar{k}_m$ does not have a common root in $\C$ by Lemma \ref{lem-deg}, $q$ cannot be a constant and thus, $\rest{\bar{h}_m}{T=q}\neq 0$ in $\C(Z)$.
Therefore, we obtain an equation in $\C(Z)$
\[ \rest{r_{m}}{T=q} = 1 - y_0^{-1} Z \]
for $r_{m} = (T\cdot \bar{k}_{m}) / \bar{h}_{m} \in \C (T)$.
We can compute that
\[ \textstyle{ \deg r_{m}\cdot \deg q =  \deg \left( \rest{r_{m}}{T=q}\right) = \deg ( 1 - y_0^{-1} Z) =1. }\]
For the first equality, see Remark \ref{rem-proj}.
However, $\deg r_{m}\geq 2$ by Lemma \ref{lem-deg},  and we get a contradiction.

By Hilbert's irreduciblity theorem (see e.g. \cite[Chapter 9, Corollary 2.4]{Lang}, which shows that $\Q$ is a Hilbertian field), there are infinitely many $Z_0\in \Q$ such that all of the specializations
\[  \rest{p_1}{Z=Z_0} ,\dots ,  \rest{p_e}{Z=Z_0}  \in \Q[T]\]
are irreducible in $\Q[T]$.
We note that the set
$\bigcup_{i=1}^e \{Z_0 \in \Q \mid \deg (\rest{p_i}{Z=Z_0}) \leq 1 \}$
is finite from the above observation of the degree of $p_1,\dots ,p_m$ in $T$.
In addition, for the fixed rational number $y_0\neq 1$, the set $\{ r \in \Q^* \mid r^n =y_0 \text{ for some }n\in \Z \}$ is also finite.
As a consequence, there are infinitely many $Z_0\in \Q^* \setminus \{ r\in \Q^* \mid r^n =y_0 \text{ for some }n\in \Z \}$ such that
$\rest{p_i}{Z=Z_0}\in \Q[T]$ is irreducible and $\deg (\rest{p_i}{Z=Z_0}) \geq 2$ for every $i\in \{1,\dots ,e\}$, which means that
$\rest{P_m}{\mu=y_0,U=Z_0}$ has no root in $\Q$.
Therefore, the relation (\ref{non-inclusion}) holds.

It remains to show the case $m\leq -2$.
Let $\varphi \in \Ham_c(T^*\R^3)$ and suppose that $\varphi(L_{K_0})$ intersects $\R^3$ cleanly along the unknot $K_1$.
Take a Hamiltonian $H\colon T^*\R^3\times [0,1] \to \R$ such that $\varphi = \varphi^1_H$.
Consider an involution
\[ \iota \colon T^*\R^3 \to T^*\R^3 \colon (q,p)= ((q_1,q_2,q_3),(p_1,p_2,p_3)) \mapsto ((q_1,q_2,-q_3),(p_1,p_2,-p_3)) \]
which preserves $\lambda_{\R^3}$. 
Then, $\iota \circ \varphi^t_H \circ \iota = \varphi^t_{\overline{H}}$ for $\overline{H} \coloneqq H\circ (\iota\times \id_{[0,1]})$. 
Moreover, for the knots
\[ \overline{K_i} \coloneqq \{ (q_1,q_2, -q_3)\in \R^3 \mid (q_1,q_2, q_3) \in K_i\} \text{ for }i=0,1,\]
$ \varphi^1_{\overline{H}} (L_{\overline{K_0}})$ intersects $\R^3 = \iota (\R^3)$ cleanly along $\overline{K_1}$. Since $\overline{K_1}$ is the unknot and $\overline{K_0}$ is the connected sum of the mirror of $K$ and $T_{(2, -2m-1)}$, 
this contradicts to the result from the previous case.
\end{proof}

\subsection{Proof of Theorem \ref{thm-1} for $K'\# 4_1$}\label{subsec-thm2}

Fix any $B\in B_n$ whose closure gives a knot $K$. We consider a braid on $(n+3)$ strands
\[B^{\#} \coloneqq \sigma_n \left( \sigma_{n+1} (\sigma_{n+2})^{-1} \sigma_{n+1}(\sigma_{n+2})^{-1}\right) B \in B_{n+3}.\]
Then, the closure of $B^{\#}$ is the connected sum of $K$ and the figure-eight knot $4_1$.

\begin{prop}\label{prop-J-eight}
We define $P\in \Z[\mu^{\pm}, U^{\pm}][T]$  by
\begin{align}\label{poly-P}
\begin{split}
P (T) \coloneqq & -\mu^{-1} (1-\mu) (U-\mu) + (-1+2\mu -U-\mu^2 U^{-1}) T -\mu T^2 \\
 & + \left( \mu^{-1}(1-\mu) + (2- \mu^{-1} -\mu U^{-1} )T - T^2 \right) \left( (1-\mu)(U-\mu) + \mu U T \right) .
\end{split}
\end{align}
Then, for any field $\bfk$,
\[\pi_{y,Z} (V_{\bfk}(K\#4_1)) \subset \{ (y,Z)\in (\bfk^*)^2 \mid \rest{P}{\mu=y,U=Z} \text{ has a root in }\bfk \}.\]
\end{prop}
\begin{proof}
First, let us compute $\phi_{B^{\#}} (a_{n+2,*})$ and $\phi_{B^{\#}}(a_{n+3,*})$.
Since $\phi_{\sigma_k}(a_{n+2,*}) = a_{n+2,*}$ and  $\phi_{\sigma_k}(a_{n+3,*}) = a_{n+3,*}$ for every $k\in \{1,\dots ,n-1\}$, we have $\phi_B(a_{n+2,*}) = a_{n+2,*}$ and $\phi_B(a_{n+3,*}) = a_{n+3,*}$.
Referring to (\ref{phi-sigma}), we compute that
\[ (\phi_{\sigma_{n+2}})^{-1} \colon
\begin{cases}
 a_{n+1,*}\mapsto a_{n+1,*},\\
 a_{n+2,*} \mapsto a_{n+3,*} , \\
  a_{n+3,*}\mapsto  a_{n+2,*}-a_{n+2,n+3} a_{n+3,*}  , \\
  a_{n+1,n+3} \mapsto a_{n+1,n+2}-a_{n+1,n+3}a_{n+3,n+2} ,
 \end{cases} \]
and
\[\xymatrix@R=5pt@C=13pt{
a_{n+2,*} \ar@{|->}[r]^-{(\phi_{\sigma_{n+2}})^{-1} } &  a_{n+3,*}  &  a_{n+3,*}  \ar@{|->}[r]^-{(\phi_{\sigma_{n+2}})^{-1} } & a_{n+2,*}-a_{n+2,n+3} a_{n+3,*}   \\
\ar@{|->}[r]^-{\phi_{\sigma_{n+1}}} & a_{n+3,*}  &  \ar@{|->}[r]^-{\phi_{\sigma_{n+1}}} & a_{n+1,*}-a_{n+1,n+3}a_{n+3,*}  \\
\ar@{|->}[r]^-{(\phi_{\sigma_{n+2}})^{-1} }  & a_{n+2,*} -a_{n+2,n+3} a_{n+3,*} &  \ar@{|->}[r]^-{(\phi_{\sigma_{n+2}})^{-1} } & a_{n+1,*} - M'' (a_{n+2,*}-a_{n+2,n+3}a_{n+3,*})  \\
\ar@{|->}[r]^-{\phi_{\sigma_{n+1}}}  & a_{n+1,*}-a_{n+1,n+3} a_{n+3, *} &  \ar@{|->}[r]^-{\phi_{\sigma_{n+1}}} &
{\begin{array}{l} a_{n+2,*}-a_{n+2,n+1}a_{n+1,*} \\ - M' (a_{n+1,*}-a_{n+1,n+3}a_{n+3,*}) \end{array} }\\
\ar@{|->}[r]^-{\phi_{\sigma_n}} & a_{n,*} - a_{n,n+3} a_{n+3,*} , & \ar@{|->}[r]^-{\phi_{\sigma_n}}  & a_{n+2,*}-a_{n+2,n}a_{n,*} - M(a_{n,*}-a_{n,n+3}a_{n+3,*}) ,
}\]
where
\begin{align*}
 M'' &\coloneqq (\phi_{\sigma_{n+2}})^{-1}(a_{n+1,n+3}) = a_{n+1,n+2} - a_{n+1,n+3}a_{n+3,n+2}, \\
 M' & \coloneqq \phi_{\sigma_{n+1}} (M'') = -a_{n+2,n+1} - (a_{n+2,n+3}-a_{n+2,n+1}a_{n+1,n+3})a_{n+3,n+1} , \\
 M & \coloneqq \phi_{\sigma_n}(M') =  -a_{n+2,n} - (a_{n+2,n+3}-a_{n+2,n}a_{n,n+3})a_{n+3,n} .
 \end{align*}
As a consequence, we obtain
\begin{align*}
\phi_{B^{\#}}(a_{n+2,*}) & =  a_{n,*} - a_{n,n+3} a_{n+3,*} , \\ 
\phi_{B^{\#}}(a_{n+3,*}) &= (-a_{n+2,n}-M)a_{n,*} + a_{n+2,*} + (M a_{n,n+3})a_{n+3,*},
\end{align*}
for 
\begin{align}
M = -a_{n+2,n}(1-a_{n,n+3}a_{n+3,n}) -a_{n+2,n+3}a_{n+3,n} .
\end{align}

The $(n+2)$-th and $(n+3)$-th row of $\boldsymbol{\Phi}^L_{B^{\#}}$ are given by
\[\begin{pmatrix}
0 & \cdots & 0 & 1 & 0 & 0 & -a_{n,n+3}  \\
0 & \cdots & 0 & -a_{n+2,n}-M & 0 & 1 & M a_{n,n+3}
\end{pmatrix} . \]
Let $c_{i,j}$ denote the $(i,j)$-th entry of $\widehat{\bold{A}} - \boldsymbol{\Lambda} \cdot \boldsymbol{\Phi}^L_{B^{\#}} \cdot \bold{A}$. Then,
\begin{align*}
c_{n+2,n} &= \widehat{\bold{A}}_{n+2,n} - (1\cdot \bold{A}_{n,n} - a_{n,n+3} \bold{A}_{n+3,n} )  \\
& = -\mu a_{n+2,n} - (1-\mu ) - \mu a_{n,n+3} a_{n+3,n} , \\
c_{n+2,n+3} & = \widehat{\bold{A}}_{n+2,n+3} - (1\cdot \bold{A}_{n,n+3} - a_{n,n+3} \bold{A}_{n+3,n+3}) \\
& =U a_{n+2,n+3}- a_{n,n+3} + (1-\mu) a_{n,n+3} \\
& =U a_{n+2,n+3} - \mu a_{n,n+3} , \\
c_{n+3,n} & = \widehat{\bold{A}}_{n+3,n} - ( (  -a_{n+2,n}-M) \bold{A}_{n,n}  + 1\cdot \bold{A}_{n+2,n} + Ma_{n,n+3} \bold{A}_{n+3,n}) \\
& =-\mu a_{n+3,n}+ (1-\mu) ( a_{n+2,n}+M ) + \mu a_{n+2,n} + \mu M a_{n,n+3} a_{n+3,n}  \\
& = -\mu a_{n+3,n} + a_{n+2,n} + M (1-\mu +\mu a_{n,n+3}a_{n+3,n}), \\
c_{n+3, n+3} & = \widehat{\bold{A}}_{n+3,n+3} - ( (  -a_{n+2,n}-M) \bold{A}_{n,n+3} + 1\cdot \bold{A}_{n+2,n+3} + Ma_{n,n+3} \bold{A}_{n+3,n+3}) \\
& =(U-\mu) +  (a_{n+2,n}+M)a_{n,n+3} - a_{n+2,n+3} - (1-\mu) M a_{n,n+3} \\
& = (U-\mu)  +  a_{n+2,n}a_{n,n+3} - a_{n+2,n+3} + \mu M a_{n,n+3} .
\end{align*}
Let us denote $X= a_{n,n+3}$ and $Y=a_{n+3,n}$.
The ideal $\mathcal{I}_{B^{\#}}$ contains $\pi_{n+3}(c_{i,j})$ for every $i,j\in \{1,\dots ,n+3\}$, where $\pi_{n+3}$ is defined as (\ref{map-pi}). In particular,  $\mathcal{I}_{B^{\#}}$ contains
\begin{align*}
\mu^{-1}\cdot \pi_{n+3}(c_{n+2,n}) & = - a_{n+2,n} - \mu^{-1}+1 - XY , \\
U^{-1} \cdot \pi_{n+3}(c_{n+2,n+3}) & = a_{n+2,n+3} -\mu U^{-1} X.
\end{align*}
Replacing $a_{n+2,n}$ with $-\mu^{-1}+1 -XY$ and $a_{n+2,n+3}$ with $\mu U^{-1} X$, we obtain from $\pi_{n+3}(M)$
\begin{align*}
\bar{M} & \coloneqq - ( -\mu^{-1}+1 -XY ) (1 -XY) - \mu U^{-1} XY \\
& = \mu^{-1}(1-\mu) +  ( 2- \mu^{-1} -\mu U^{-1} )XY -(XY)^2,
\end{align*}
and by the same replacement, we obtain from $\pi_{n+3}(c_{n+3,n})$ and $\pi_{n+3}(c_{n+3,n+3})$
\begin{align*}
F &\coloneqq  -\mu Y + ( -\mu^{-1}+1 -XY) + \bar{M}(1-\mu + \mu XY),  \\
G & \coloneqq  (U-\mu) + (  -\mu^{-1}+1 -XY ) X - \mu U^{-1} X + \mu\bar{M} X  \\
& = (U-\mu) +  (- \mu^{-1}+1-\mu U^{-1} - XY ) X+ \mu\bar{M} X .
\end{align*}
Then, both $F$ and $G$ are contained in $\mathcal{I}_{B^{\#}}$.

We focus on the first terms $-\mu Y$ and $(U-\mu)$ of $F$ and $G$. Then, we compute
\begin{align*}
 (U-\mu) \cdot F + \mu Y \cdot G =  & (U-\mu) (  -\mu^{-1}+1 -XY) + (U-\mu)\bar{M}(1-\mu + \mu XY) \\
 & + \mu ( - \mu^{-1}+1-\mu U^{-1} - XY ) XY  + \mu^2\bar{M} XY \\
 = &  -\mu^{-1} (1-\mu) (U-\mu) + (-1+2\mu -U-\mu^2 U^{-1}) XY -\mu (XY)^2 \\
 & + \bar{M} \left( (1-\mu)(U-\mu) + \mu U XY \right) .
\end{align*}
If we define $P\in \Z[\mu^{\pm},U][T]$ by (\ref{poly-P}), then
\[ (U-\mu) \cdot F + \mu Y \cdot G = P(XY),\] 
which implies that $P(XY)$ is contained in $\mathcal{I}_{B^{\#}}$.
Therefore, we have a well-defined unital $R$-algebra map
\[ J\colon R[T]/ (P) \to R[a_{i,j}\mid 1\leq i,j\leq n+3,\ i\neq j] /\mathcal{I}_{B^{\#}} = HC^{\ab}_0(K\# 4_1) \]
which maps $T $ to $XY  = a_{n,n+3} a_{n+3,n} $.

The rest of argument is parallel to Proposition \ref{prop-Jm}.
Given a $\bold{k}$-valued augmentation $\epsilon$ of $\mathcal{A}^{\knot}_*(K\# 4_1)$, it induces a unital ring homomorphism from  $HC^{\ab}_0(K\#4_1)$ to $\bfk$.
By precomposing $J$, we obtain a unital ring homomorphism $ R[T] / (P)  \to \bold{k}$ which sends $\mu$ to $\epsilon(\mu)$ and $U$ to $\epsilon(U)$.
This means that $\rest{P}{\mu=\epsilon(\mu),U=\epsilon(U)}\in \bold{k}[T]$ has a root in $\bold{k}$.
Now, the proposition follows from the definition of $V_{\bold{k}}(K\# 4_1)$. 
\end{proof}

\begin{rem}
Similar to $\rest{P_m}{U=1}$ in Remark \ref{rem-reduction}, consider a substitution by $U=1$. Then
\begin{align*}
\rest{P}{U=1}(T) =&  -\mu^{-1} (1-\mu)^2 + (-2+2\mu -\mu^2 ) T -\mu T^2 \\
 & + \left( \mu^{-1}(1-\mu) + ( 2- \mu^{-1} -\mu )T - T^2 \right) \left( (1-\mu)^2 + \mu  T \right) \\
= & \  ( -\mu^{-1} -T )  \left( (1-\mu)^2 + \mu  T \right) \\
 & + \left( \mu^{-1}(1-\mu) + (2- \mu^{-1} -\mu )T - T^2 \right) \left( (1-\mu)^2 + \mu  T \right) .
\end{align*}
Therefore, $\rest{P}{U=1}\in \Z[\mu^{\pm}][T]$ has a root $T= - \mu^{-1}(1-\mu)^2$.
\end{rem}

\begin{thm}\label{thm-eight}
Let $K_0 = K \# 4_1$, where $K$ is an arbitrary knot in $\R^3$.
Then, there is no $\varphi \in \Ham_c(T^*\R^3)$ such that $\varphi(L_{K_0})$ and $\R^3$ have a clean intersection along the unknot in $\R^3$.
\end{thm}
\begin{proof}
Let us substitute $\mu$ in $P$ by $-1$. Then,
\begin{align*}
 \rest{P}{\mu=-1} & = 2 (U+1) + (-3 -U-U^{-1})T + T^2 + ( -2  + ( 3+ U^{-1})T -T^2 )   ) ( 2(U+1) - UT  )  \\
 & = -2(U+1)  + ( 5 + 7U + U^{-1})T + (-5U-2 )T^2 +UT^3.
 \end{align*}
We substitute $U$ by $-\frac{1}{2}$, then
\[ \rest{P}{\mu=-1,U=-\frac{1}{2}} = \textstyle{ -1  -\frac{1}{2} T + \frac{1}{2}T^2 -\frac{1}{2}T^3.} \]
Note that $(-1,-\frac{1}{2}) $ is in  $ \{ (y,Z) \mid y\neq Z^n\text{ for all }n\in \Z \}$.
If we can check that $ \rest{P}{\mu=-1,U=-\frac{1}{2}}$ has no root in $\Q$, the theorem immediately follows from Proposition \ref{prop-unknot-aug} and Proposition \ref{prop-J-eight}.

Multiplying by $2$, we obtain a polynomial with integer coefficients
\[ \bar{P}\coloneqq 2\cdot \rest{P}{\mu=-1,U=-\frac{1}{2}} =  -2 - T +T^2 - T^3 .\]
By the rational root theorem, the only possible rational roots of $\bar{P}$ are  $1,-1,2$ and $-2$.
%
One can check that
\[\begin{array}{llll} \bar{P}( 1 ) = -3 , & \bar{P}(-1) = 1 , & \bar{P}(2) = -8, & \bar{P}(-2) = 12 . \end{array}\]
Therefore, we can conclude that $\bar{P}$ does not have a root in $\Q$.
\end{proof}

\appendix

\section{Proof of Proposition \ref{prop-Floer}}\label{sec-Floer}

Let $K_0$ be an oriented knot in $\R^3$ and consider its conormal bundle $L_{K_0}$.
The Lagrangian submanifolds $\R^3$ and $L_{K_0}$ of $T^*\R^3$ are oriented as in Subsection \ref{subsubsec-CE-conormal}.

For any $\psi \in \Ham_c(T^*\R^3)$, we will introduce a $\Z/2$-graded Floer cochain complex \[(CF^*(\psi(\R^3) , L_{K_0} ;H ), d_J)\]
with coefficients in the Laurent polynomial ring $\F_2 [\lambda^{\pm}]$.
It is a slight enhancement of the Floer cochain complex over $\F_2$ in \cite[Section 8]{Seidel}.

We choose a Hamiltonian $H \colon T^*\R^3\times [0,1] \to \R$ with  compact support such that the Lagrangian submanifold
\[P_H\coloneqq (\varphi^1_{H}\circ \psi) (\R^3)\]
intersects $L_{K_0}$ transversely. $P_H$ is oriented so that the diffeomorphism $\varphi^1_H \circ \psi \colon \R^3 \to P_H$ preserves orientations.
From the orientations of $P_H$ and $L_{K_0}$, a $\Z/2$-grading $|x|\in \Z/2$ is defined for every $x\in P_H \cap L_{K_0}$. See \cite[Example 2.7, Section 2.d]{Seidel-gr}.
Note that the grading is shifted by $1$ if we use the opposite orientation of $P_H$.
%
We also choose an almost complex structure $J$ such that  $d \lambda_{\R^3}(\cdot, J\cdot)$ is a Riemannian metric on $T^*\R^3$ and $J (\partial_{p_i}) = \partial_{q_i}$ for $i\in \{1,2,3\}$ outside a compact subset of $T^*\R^3$.

Fix a base point $z_0\in L_{K_0}$.
For any $x \in  P_H \cap L_{K_0}$, we choose a path $c_x\colon [0,1] \to L_{K_0}$ such that $c_x(0)=z_0$ and $c_x(1)= x$.
Let $c_x^{-1}$ denote the path $[0,1] \to L_{K_0} \colon t\mapsto c_x(1-t)$.
For every $x,y \in P_H \cap L_{K_0}$ and a homology class $A\in H_1(L_{K_0})$, we define $\mathcal{M}_{(H,J),A}(x,y)$ to be the space of maps $u\colon \R\times [0,1]\to T^*\R^3$ satisfying:
\begin{itemize}
\item $\partial_s u (s,t) + J_{u (s,t)} (\partial_t u (s,t) - (X^t_{H})_{u(s,t)})=0$ for every $(s,t) \in \R\times [0,1]$.
\item $u(\R\times \{0\})\subset \psi(\R^3) $ and $ u(\R\times \{1\}) \subset L_{K_0}$.
\item $\lim_{s\to -\infty}u(s,t) = \varphi^t_{H} ((\varphi^1_H)^{-1}(x))$ and $\lim_{s\to \infty} u(s,t)= \varphi_{H}^t((\varphi^1_H)^{-1}(y))$ $C^{\infty}$-uniformly on $t\in [0,1]$.
\item $A\in H_1(L_{K_0})$ is represented by a loop based at $z_0$ obtained by concatenating the three paths in $L_{K_0}$: $c_x$, the path $[-\infty,\infty] \to L_{K_0}\colon s\mapsto u(s,1)$ from $x$ to $y$, and $c_y^{-1}$.
\end{itemize}
Figure \ref{figure-strip} illustrates $u(\R\times [0,1])\subset T^*\R^3$ and the three paths in $L_{K_0}$.
The group $\R$ acts on $\mathcal{M}_{(H,J),A}(x,y)$ by $(u\cdot \sigma)(s,t) = u(s + \sigma ,t)$ for every $\sigma\in \R$.
\begin{figure}
\centering
\begin{overpic}[height=4cm]{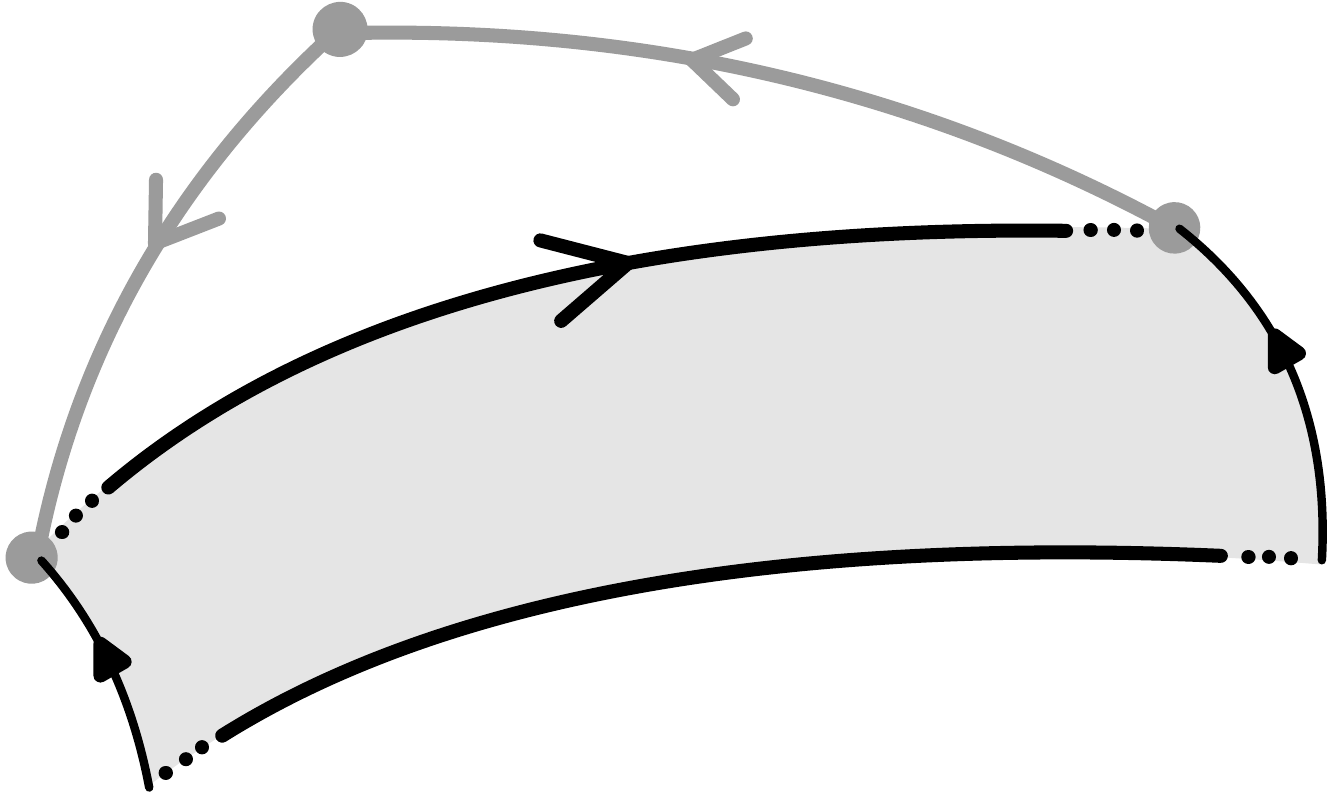}
\put(40,25){$u (\R\times [0,1] )$}
\put(31,43.5){$u(\R\times \{1\})$}
\put(56,57){$c_y^{-1}$}
\put(5.5,42){$c_x$}
\put(24,61){$z_0$}
\put(-4,16.5){$x$}
\put(91,44){$y$}
\put(50,11.5){$u(\R\times \{0\})$}
\end{overpic}
\caption{
The shaded region is the image of $u$.
$c_x$ and $c_y$ are paths in $L_{K_0}$ and $u(\R\times \{1\})$ is contained in $ L_{K_0}$. $u(\R\times \{0\})$ is contained in $\psi(\R^3)$.}\label{figure-strip}
\end{figure}

To define the $\Z/2$-graded Floer cochain complex, we refer to \cite[Section 2.f]{Seidel-gr}.
We take the pair $(H,J)$ to be \textit{regular} so that the moduli space $\mathcal{M}_{(H,J),A}(x,y)/\R$ is cut out transversely for every $x,y \in P_H \cap L_{K_0}$ such that $x\neq y$.
The dimension of  $\coprod_{A\in H_1(L_{K_0})}  \mathcal{M}_{(H,J),A}(x,y)/\R$ is equal to $|x|-|y|-1$ modulo $2$.
Moreover,
by the Gromov compactness, the union of $0$-dimensional components of $\coprod_{A\in H_1(L_{K_0})}  \mathcal{M}_{(H,J),A}(x,y)/\R$ is a finite set.
We define $n_{(H,J),A}(x,y) \in \F_2$ to be the modulo $2$ count of the number of the $0$-dimensional components of $\mathcal{M}_{(H,J),A}(x,y)/\R$.

We define $CF^*(\psi(\R^3), L_{K_0};H)$ to be the $\Z/2$-graded $\F_2[\lambda^{\pm}]$-module freely generated by the set $P_H\cap L_{K_0}$ whose grading is given by $|\cdot | \colon P_H\cap L_{K_0} \to \Z/2$. The $\F_2[\lambda^{\pm}]$-linear map $d_J$ of degree $1$ is defined by
\[ d_J (y) \coloneqq \sum_{m=-\infty}^{\infty} \sum_{|x|\equiv |y| +1 \mod 2} \left( n_{(H,J),m[K_0]}(x,y) \right) \lambda^m \cdot x  \]
for every $y\in P_H\cap L_{K_0}$.
By a standard argument using the Gromov compactification of the $1$-dimensional components of $\mathcal{M}_{(H,J),A}(x,y)/\R$ for each $A\in H_1(L_{K_0})$, we can prove that $d_J\circ d_J =0$.

From this cochain complex, we obtain a cohomology group
\[HF^*(\psi (\R^3), L_{K_0};(H,J)) \coloneqq \ker d_J / \Im d_J.\]
We claim that its isomorphism class as a $\Z/2$-graded $\F_2[\lambda^{\pm}]$-module is independent of the choice of the paths $\{c_x\mid x\in P_H\cap L_{K_0}\}$.
To see this, we choose another path $c'_x$ instead of $c_x$ for every $x\in P_H\cap L_{K_0}$ and define $d'_J$ by using $\{c'_x\mid x\in P_H\cap L_{K_0}\}$ in the same way as $d_J$. Take $m_x \in \Z$ such that $m_x[K_0] \in H_1(L_{K_0})$ is represented by a loop obtained by concatenating $c'_x$ and $c_x^{-1}$. Then we can define an $\F_2[\lambda^{\pm}]$-linear isomorphism
\[
\begin{array}{cc} \Psi \colon CF^*(\psi(\R^3),L_{K_0};H) \to CF^* (\psi(\R^3),L_{K_0};H) \colon x \mapsto \lambda^{m_x}x & \text{for }x\in P_H \cap L_{K_0}, \end{array}\]
for which we have $d_J\circ \Psi = \Psi \circ d'_J$.
This proves the claim.

Moreover, by a standard argument using a homotopy from $\psi$ to $\id_{T^*\R^3}$ in $\Ham_c(T^*\R^3)$, we can show the invariance of the Floer cohomology under Hamiltonian isotopies, which we state as follows.
\begin{thm}\label{thm-floer-isom}
If we choose a pair $(H_0,J_0)$ as above to define $(CF^*(\R^3,L_{K_0};H_0),d_{J_0})$ (i.e. the case $\psi= \id_{T^*\R^3}$),
then there exists an isomorphism of $\Z/2$-graded $\F_2[\lambda^{\pm}]$-modules
\[
HF^*(\psi(\R^3),L_{K_0};(H,J)) \cong HF^*(\R^3 , L_{K_0} ; (H_0,J_0)) .
\]
\end{thm}


Now, we prove Proposition \ref{prop-Floer} about a clean Lagrangian intersection along a circle.

\begin{proof}[Proof of Proposition \ref{prop-Floer}]
For $\varphi \in \Ham_c(T^*\R^3)$ in the statement of Proposition \ref{prop-Floer}, let us set $\psi \coloneqq \varphi^{-1}$. Then, we can change the setup as follows:
 $\psi(\R^3)$ and $L_{K_0}$ have a clean intersection along $\psi(K_1)$, where $K_1$ is an oriented knot in $\R^3$. 
We may assume that for sufficiently small $\epsilon>0$, $\psi ( L_{K_1}\cap D^*_{\epsilon}\R^3) \subset L_{K_0}$.
We fix an orientation of $\psi(K_1)$ via the diffeomorphism $\rest{\psi}{K_1} \colon K_1 \to \psi(K_1)$ and define $a\in \Z$ by $[\psi(K_1)] = a [K_0] \in H_1(L_{K_0})$.
Let $(-1)^{\sigma}$ denote the sign of orientation of the embedding $\rest{\psi}{L_{K_1}\cap D^*_{\epsilon}\R^3} \colon L_{K_1}\cap D^*_{\epsilon}\R^3 \to L_{K_0}$.
Then, the claim of the proposition is equivalent to that $a \in \{1,-1\}$ and $\sigma\equiv 0 \mod 2$.
Since $\sigma\equiv 0 \mod 2$ is already shown in \cite[Lemma 4.4]{O}, we only give a proof for $a\in \{1,-1\}$ (see also Remark \ref{rem-sign} below).

We choose a Mose function $f\colon \psi(K_1) \to \R$ which has only two critical points $x_0,x_1\in \psi(K_1)$ such that $f$ takes its minimum at $x_0$ and its maximum at $x_1$. We also choose a Riemannian metric $g$ on $\psi(K_1)$. By using $g$, we obtain the gradient vector field $V_f$ for $f$ by $g(V_f,\cdot) = df$. Then, we consider the set $\mathcal{T}_{(f,g)}(x_0,x_1)$ of trajectories $\gamma\colon \R\to \psi(K_1)$ such that:
\begin{itemize}
\item $\frac{d \gamma}{ds}(s) = (V_f)_{\gamma(s)}$ for every $s\in \R$.
\item $\lim_{s\to -\infty}\gamma (s) = x_0$ and $\lim_{s \to \infty} \gamma(s) = x_1$.
\end{itemize}
Since $\psi(\R^3)$ and $L_{K_0}$ intersects cleanly along $\psi(K_1)$ and $\psi(K_1)$ is connected, we can apply \cite[Theorem 3.4.11, Lemma 3.4.12]{P} to get the following results:
From the pair $(f,g)$ with $||f ||_{C^1}$ sufficiently small,
we can construct
a Hamiltonian $H_f \colon T^*\R^3 \to \R$ supported in a neighborhood of $\psi(K_1)$ and an almost complex structure $J_g$ on $T^*\R^3$ such that the pair $(H_f,J_g)$ is regular and the following hold:  $P_{H_f} \cap L_{K_0} = \{x_0,x_1\}$ and
there is a bijection
\[ \Phi \colon \coprod_{A \in H_1(L_{K_0})} \M_{(H_f,J_g),A}(x_0,x_1) \to \mathcal{T}_{(f,g)}(x_0,x_1) \colon u \mapsto u(\cdot,1) .\]

It is clear that $\mathcal{T}_{(f,g)}(x_0,x_1)$ has only two elements.
Let us label them as $\gamma_1$ and $\gamma_2$ so that $\frac{d \gamma_2}{ds}(s)$ is a positive vector with respect to the orientation of $\psi(K_1)$.
We tale $k\in \Z$ such that $\Phi^{-1}(\gamma_1)$ belongs to $\M_{(H_f,J_g),k[K_0]}(x_0,x_1)$. ($k$ depends on the choice of the paths $c_{x_0}$ and $c_{x_1}$.) Then, $\Phi^{-1}(\gamma_2)$ belongs to $\M_{(H_f,J_g),A}(x_0,x_1)$ for the homology class
\[A= k[K_0] + [\psi(K_1)] = (k+a)[K_0]\]
since $[\psi(K_1)] \in H_1(L_{K_0})$ is represented by a loop obtained by concatenating the two paths: $[-\infty,\infty] \to L_{K_0}\colon s \mapsto \gamma_1(-s)$ and $[-\infty,\infty] \to L_{K_0}\colon s \mapsto \gamma_2(s)$.
From this observation, the differential $d_{J_g}\colon \F_2[\lambda^{\pm}] \cdot  x_1 \to \F_2[\lambda^{\pm}] \cdot x_0 $ is given by
\[d_{J_g} (x_1) =  (\lambda^k + \lambda^{k+a})x_0 .\]
Therefore,
$HF^{*}(\psi(\R^3),L_{K_0};(H_f,J_g))$ is computed as follow (we omit writing $(H_f,J_g)$):
\[ \begin{array}{lll} HF^{0+\sigma_{\psi}}(\psi(\R^3),L_{K_0}) = 0, &  HF^{1+\sigma_{\psi}}(\psi(\R^3),L_{K_0}) \cong \F_2[\lambda^{\pm}]/ (1+\lambda^a) & \text{ if }a\neq 0, \\
HF^{0+\sigma_{\psi}} (\psi(\R^3),L_{K_0}) \cong \F_2[\lambda^{\pm}], &  HF^{1+\sigma_{\psi}} (\psi(\R^3),L_{K_0}) \cong \F_2[\lambda^{\pm}] & \text{ if }a = 0.
\end{array}\]
Here, $\sigma_{\psi}\in \Z/2$ is the shift of degrees.
As $\F_2$-vector spaces,
\begin{align}\label{HF-psi}
\begin{array}{lll} \dim_{\F_2} HF^{0+\sigma_{\psi}}(\psi(\R^3),L_{K_0})  = 0, & \dim_{\F_2} HF^{1+\sigma_{\psi}} (\psi(\R^3),L_{K_0})  = |a| & \text{ if }a\neq 0, \\
\dim_{\F_2} HF^{0+\sigma_{\psi}} (\psi(\R^3),L_{K_0})  = \infty, & \dim_{\F_2} HF^{1+\sigma_{\psi}}(\psi(\R^3),L_{K_0})  =\infty & \text{ if }a = 0. \end{array}
\end{align}
If we consider a special case where $\psi$ is the identity map $\id = \id_{T^*\R^3}$, then $K_1=\psi(K_1) = K_0$ and we have $a=\pm 1$. Therefore,
\begin{align}\label{HF-id}
\begin{array}{ll} \dim_{\F_2} HF^{0+\sigma_{\id}}(\R^3,L_{K_0}) = 0, & \dim_{\F_2} HF^{1+ \sigma_{\id}}(\R^3,L_{K_0}) = 1 . \end{array}
\end{align}
From (\ref{HF-psi}), (\ref{HF-id}) and Theorem \ref{thm-floer-isom}, it follows that
$\sigma_{\psi} \equiv \sigma_{\id} \mod 2$ and $a\in \{ 1 , -1\}$.
\end{proof}

\begin{rem}\label{rem-sign}
The fact $\sigma \equiv 0 \mod 2$ was proved in \cite[Lemma 4.4]{O} by introducing $\Z$-gradings of Lagrangian submanifolds which are lifts of orientations of Lagrangian submanifolds.
Analogous to \cite[Remark 2.9, Lemma 4.4]{O},
by an argument on $\Z/2$-grading in the present case, we can deduce $\sigma \equiv 0 \mod 2$ from $\sigma_{\psi} \equiv \sigma_{\id} \mod 2$.
\end{rem}

\bibliography{reference.bib}

\end{document}